\documentclass[12pt,oneside]{amsart}
\usepackage[utf8]{inputenc}
\usepackage[T1]{fontenc}
\usepackage[english]{babel}

\usepackage{amsthm}
\usepackage{amsfonts}
\usepackage{amsmath}
\usepackage{amssymb}
\usepackage{mathtools}

\usepackage[a4paper]{geometry}
\usepackage{enumitem}
\setlist[enumerate,1]{label = (\roman*)}

\theoremstyle{definition}
\newtheorem{definition}{Definition}[section]
\newtheorem{theorem}[definition]{Theorem}
\newtheorem*{theorem*}{Theorem}
\newtheorem{lemma}[definition]{Lemma}
\newtheorem{corollary}[definition]{Corollary}

\newtheorem{question}[definition]{Question}
\newtheorem*{question*}{Question}
\theoremstyle{remark}
\newtheorem{remark}[definition]{Remark}

\DeclareMathOperator{\dom}{\textrm{dom}}
\DeclareMathOperator{\ran}{\textrm{ran}}

\newcommand{\On}{\textrm{On}}
\newcommand{\R}{\mathbb{R}}
\newcommand{\Q}{\mathbb{Q}}

\newcommand{\concat}[2]{{#1}^\frown{#2}}

\newcommand{\wins}{\mathbin{\raisebox{0.8\depth}{$\uparrow$}}}
\newcommand{\doesnotwin}{\mathbin{{\hspace*{-0.125em}\not\hspace*{0.125em}\uparrow}}}
\newcommand{\playerone}{\mathrm{\mathbf{I}}}
\newcommand{\playertwo}{\mathrm{\mathbf{II}}}

\newcommand{\watershed}[3]{\operatorname{SW}_{#1}^{#2}\left( #3 \right)}
\DeclareMathOperator{\scottrank}{F\text{-}SR}

\newcommand{\EFR}[3]{\mathrm{R\text{-}EF}_{#1}^{#2}\left(#3\right)}
\newcommand{\EFDR}[3]{\mathrm{R\text{-}EFD}_{#1}^{#2}\left(#3\right)}
\newcommand{\EF}[3]{\mathrm{F\text{-}EF}_{#1}^{#2}\left(#3\right)}
\newcommand{\EFD}[3]{\mathrm{F\text{-}EFD}_{#1}^{#2}\left(#3\right)}
\newcommand{\watershedr}[2]{\operatorname{R\text{-}SW}_{#1}\left( #2 \right)}

\DeclareMathOperator{\scottrankr}{R\text{-}SR}

\newcommand{\model}[1]{\mathfrak{#1}}

\newcommand{\dense}[1]{\mathcal{D}_{#1}}
\newcommand{\A}{\model{A}}
\newcommand{\B}{\model{B}}

\DeclareMathOperator{\density}{\mathrm{density}}
\newcommand{\abs}[1]{\left| #1 \right|}
\newcommand{\open}[1]{\left( #1 \right)}

\newcommand{\norm}[1]{\left\| #1 \right\|}
\newcommand{\ceil}[1]{\left\lceil #1 \right\rceil}
\newcommand{\floor}[1]{\left\lfloor #1 \right\rfloor}

\newcommand{\Logic}[1]{\mathcal{L}_{#1 \omega}}

\newcommand{\gameformula}[3]{{}^{#1}\sigma_{#2}^{#3}}
\newcommand{\gameformular}[3]{{}^{#1}\varsigma_{#2}^{#3}}
\newcommand{\weakneg}{\mathrm{neg}}

\DeclareMathOperator{\dr}{d}
\DeclareMathOperator{\Dr}{D}
\newcommand{\good}{\mathop{\textsc{Good}}}
\newcommand{\appr}{\mathop{\textsc{Appr}}}

\newcommand{\class}{\mathcal{K}}
\newcommand{\isom}{\mathcal{I}}

\newcommand{\dbm}{\mathop{d_{\mathrm{BM}}}}
\newcommand{\dgh}{\mathop{d_{\mathrm{GH}}}}
\newcommand{\dl}{\mathop{d_{\mathrm{L}}}}
\newcommand{\dk}{\mathop{d_{\mathrm{K}}}}
\newcommand{\dH}[1]{\mathop{d_{\mathrm{H}}^{#1}}}
\newcommand{\corr}{\mathcal{R}}

\DeclareMathOperator{\innerscottrank}{F\text{-}SR_{0^+}}
\DeclareMathOperator{\innerscottrankr}{R\text{-}SR_{0^+}}
\newcommand{\LM}{L_{\mathrm{M}}}
\newcommand{\LB}{L_{\mathrm{B}}}
\newcommand{\LBb}{L^{\mathrm{b}}_{\mathrm{B}}}
\newcommand{\BS}{\mathrm{BS}}
\newcommand{\BSb}{\mathrm{BS}^{\mathrm{b}}}
\newcommand{\MS}{\mathrm{MS}}

\makeatletter
\newcommand{\dotminus}{\mathbin{\text{\@dotminus}}}
\newcommand{\@dotminus}{%
    \ooalign{\hidewidth\raise1ex\hbox{.}\hidewidth\cr$\m@th-$\cr}%
}
\makeatother

\begin{document}

    \title[Games and Scott sentences for metric structures]{Games and Scott sentences for positive distances between metric structures}
    \date{\today}
    \author{Åsa Hirvonen}
    \email{asa.hirvonen@helsinki.fi}
    \author{Joni Puljujärvi}
    \email{joni.puljujarvi@helsinki.fi}
    \address{Department of Mathematics and Statistics, University of Helsinki, P.O. Box 68 (Pietari Kalmin katu 5), 00014 Helsinki, Finland}
    \thanks{J.~Puljujärvi was partially supported by the Academy of Finland, grant 322795.}
    \keywords{Banach--Mazur distance, Banach spaces, Ehrenfeucht--Fraïssé games, Gromov--Hausdorff distance, infinitary logic, Kadets distance, linear isomorphism, metric model theory, metric spaces, Scott rank, Scott sentence}
    \subjclass[2020]{03C66, 03C75}

    \begin{abstract}
        We develop various Ehrenfeucht--Fraïssé games for distances between metric structures. We study two forms of distances: pseudometrics stemming from mapping spaces onto each other with some form of approximate isomorphism, and metrics stemming from measuring the distances between two spaces isometrically embedded into a third space. Using an infinitary version of Henson's positive bounded logic with approximations, we form Scott sentences capturing fixed distances to a given space. The Scott sentences of separable spaces are in $\Logic{\omega_1}$ for 0-distances and in $\Logic{\omega_2}$ for positive distances.
    \end{abstract}

    \maketitle

    \section{Introduction}

    This paper started as an investigation into Ehrenfeucht--Fraïssé games and Scott sentences for linear isomorphism of Banach spaces, inspired by the Scott sentences for (isometries of) metric spaces developed in~\cite{BYDNT17} and the approximate language developed by Henson in~\cite{Henson76}. As linear isomorphisms are in many cases the main isomorphism notion used in functional analysis, we wanted to study how to express it with metric infinitary logic.

    Various distances between metric spaces have been studied actively. Already Ben Yaacov et al.~\cite{BYDNT17} show that the distance notion they study capture Gromov--Hausdorff and Kadets distances. However, their Scott sentences only work for distance 0. Also, a referee to the first version of this paper pointed out to us that approximate notions of isomorphism and similar game methods had been studied also in~\cite{CDK20} and~\cite{H20}. Also in these papers the methods work for distance 0 (almost isometry), but they inspired us to generalise our setting to more general metric structures, where we still can capture also positive distances.

    We study two notions of distances between spaces: pseudometrics arising from mapping spaces onto each other by some sort of approximate isomorphisms, and metrics arising from isometrically embedding spaces into a third space and minimizing the Hausdorff distance of the images. The former ones are captured by a collection of approximate isomorphisms, the latter by relations which we, following~\cite{CDK18}, call correspondences. These two notions give rise to two different games. The isomorphism distances use a game resembling the techniques used by Ben Yaacov et al.~\cite{BYDNT17} where the second player has to play better and better moves so that the limits of moves form an $\varepsilon$-mapping in the end. In the correspondence games the second player can play ``correct'' answers at once and the moves played generate a correspondence.

    For both the isomorphism and correspondence notions, we study two forms of Ehrenfeucht--Fraïssé games, infinite games of length $\omega$ and finite but dynamical games where an ordinal clock is running downwards. In our game development we follow~\cite{Vmodels} where classical Scott sentences are presented via games. The distance is a parameter in our games, and only the 0-distance games are strictly symmetric and transitive. However, for the positive distances the games have weak symmetry and transitivity properties, and when studying almost isometries, one can use these properties for similar constructions as in the classical EF-games. Thus Scott sentences expressing almost isometry can be found in $\Logic{\omega_1}$. Due to the lack of transitivity, we were not able to prove that separable spaces would have countable Scott rank when positive distances are concerned, calling for the question:
    \begin{question*}
        Are there separable spaces (in particular Banach spaces) with Scott rank $\omega_1$ (for positive distances)?
    \end{question*}
    We can, however show that the Scott rank is at most $\omega_1$ and thus positive distances give Scott sentences in $\Logic{\omega_2}$.

    In section~\ref{Section: logic} we describe the language, approximate isomorphisms and correspondences used. Section~\ref{Section: Examples} presents our main motivating examples: linear isomorphism, bi-Lipschitz mappings, Gromov--Hausdorff and Kadets distances. Here the notions of approximation varies depending on whether we consider $\varepsilon$-isomorphisms or $\varepsilon$-correspondences. The isomorphisms work best with a multiplicative approximation, whereas correspondences use an additive approximation. In section~\ref{Section: Games} we present infinite and dynamic EF-games both for approximate isomorphisms and graded correspondence notions, and prove weak symmetry and transitivity. In section~\ref{Section: Scott Rank} we define two Scott ranks: the general for arbitrary distances, and another for almost isometry. We show that for separable structures the first is at most $\omega_1$ and the latter is countable. Finally in section \ref{Section: Scott sentences} we build Scott sentences, both for positive and 0-distances.

    \section{Logics with Approximations}\label{Section: logic}
    
    For a regular cardinal $\kappa$ and some vocabulary $L$, we define the syntax of the logic $\Logic\kappa[L]$ as follows. If $R\in L$ is an $n$-ary relation symbol and $\bar{t}$ is an $n$-tuple of first-order $L$-terms, then $R(\bar{t})$ is an atomic $L$-formula. If $\varphi$ is an $L$-formula and $v$ is a variable, then $\exists v\, \varphi$ and $\forall v\, \varphi$ are $L$-formulae. Whenever $\Phi$ is a set of $L$-formulae such that $\abs{\Phi}<\kappa$ and there is a finite set $V$ of variables such that whenever $v$ is free in some $\varphi\in\Phi$, then $v\in V$, then $\bigwedge\Phi$ and $\bigvee\Phi$ are $L$-formulae. In other words, we have the standard syntax for $\Logic{\kappa}$ but without negation or equality.

    Given $\varphi(\bar{v})$ and an $L$-structure $\A$ and $\bar{a}\in\A^{\abs{\bar{v}}}$, we define $\A\models\varphi(\bar{a})$ in the usual way.

    \begin{definition}
        \label{Definition: Vocabulary with approximations}
        Let $L$ be a vocabulary and $\class$ a class of $L$-structures, and denote by $\Phi$ the set of atomic $L$-formulae. We say that $L$ has approximations for the class $\class$ if there exist functions $\appr\colon\Phi\times D\to\Phi$ and $\weakneg\colon\Phi\to\Phi$, where $D\subseteq[0,\infty)$ is dense, contains $0$ and is closed under addition, so that the following hold.
        \begin{enumerate}[label = (\Roman*)]

            \item\label{item: Approximation}
            \begin{enumerate}[label = (\roman*)]
                \item For all $\varphi\in\Phi$ and $\varepsilon\in D$, the formulae $\varphi$ and $\appr(\varphi,\varepsilon)$ have the same variables.

                \item $\appr(\varphi,0) \equiv \varphi$.

                \item\label{item: Error in approximation is additive}
                $\appr(\appr(\varphi,\varepsilon),\delta) \equiv \appr(\varphi,\varepsilon+\delta)$ for all $\varphi\in\Phi$ and $\varepsilon,\delta\in D$.

                \item\label{item: Better approximation of a formula implies worse approximation of the formula}
                For any $\varphi\in\Phi$ and $\varepsilon\in D$,
                \[
                    \varphi \models \appr(\varphi,\varepsilon).
                \]

                \item\label{item: An atomic formula is satisfied iff each of its approximations is satisfied}
                For any $\varphi(\bar{v})\in\Phi$, $\A\in\class$ and $\bar{a}\in\A^{\abs{\bar{v}}}$, we have
                \[
                    \A\models\varphi(\bar{a}) \iff \A\models\appr(\varphi,\varepsilon)(\bar{a}) \quad\text{for all $\varepsilon\in D\setminus\{0\}$}.
                \]
            \end{enumerate}

            \item\label{item: Weak negation}
            \begin{enumerate}[label = (\roman*)]
                \item For all $\varphi\in\Phi$, the formulae $\varphi$ and $\weakneg(\varphi)$ have the same variables.

                \item $\weakneg(\weakneg(\varphi)) = \varphi$.

                \item\label{item: Weak negation property}
                For any $\varphi(\bar{v})\in\Phi$, $\A\in\class$ and $\bar{a}\in\A^{\abs{\bar{v}}}$,
                \[
                    \A\not\models\varphi(\bar{a}) \iff \A\models\weakneg(\appr(\varphi,\varepsilon))(\bar{a}) \quad\text{for some $\varepsilon\in D\setminus\{0\}$}.\hspace*{-2.5em}
                \]

                \item\label{item: Approximation of weak negation}
                For $\varphi(\bar{v})\in\Phi$, $\varepsilon,\delta\in D$, $\A\in\class$ and $\bar{a}\in\A^{\abs{\bar{v}}}$, if $\varepsilon\geq\delta>0$, then
                \[
                    \appr(\weakneg(\appr(\varphi,\varepsilon)),\delta) \equiv \weakneg(\appr(\varphi,\varepsilon-\delta)).
                \]
            \end{enumerate}

            \item\label{item: Definable metric}
            For each $\A\in\class$, there is a complete metric $d_\A$ on $\A$ such that for each $r\in D$, the set $\{(a,a')\in\A^2 \mid d(a,a')\leq r\}$ is definable by some $\varphi\in\Phi$.

        \end{enumerate}
        We also call $L$ together with the functions $\appr$ and $\weakneg$ a vocabulary with approximations with respect to $\class$. If the class $\class$ is implicitly known, we simply say that $L$ is a vocabulary with approximations.

        We call $\appr(\varphi,\varepsilon)$ the $\varepsilon$-approximation of $\varphi$ and $\weakneg(\varphi)$ the weak negation of $\varphi$. When the set $D$ is clear from the context, we drop it out from the notation for convenience, i.e. we write $\delta > 0$ instead of $\delta \in D\cap[0,\infty)$.
    \end{definition}

    \begin{definition}
        Let $L$ be a vocabulary with approximations. We extend the notion of $\varepsilon$-approximation for all $L$-formulae as follows.
        \begin{itemize}
            \item $\appr\left( \bigwedge_{i\in I}\psi, \varepsilon \right) = \bigwedge_{i\in I}\appr(\psi,\varepsilon)$,
            \item $\appr\left( \bigvee_{i\in I}\psi, \varepsilon \right) = \bigvee_{i\in I}\appr(\psi,\varepsilon)$,
            \item $\appr\left( \forall v\, \psi, \varepsilon \right) = \forall v\appr(\psi,\varepsilon)$, and
            \item $\appr\left( \exists v\, \psi, \varepsilon \right) = \exists v\appr(\psi,\varepsilon)$.
        \end{itemize}
    \end{definition}

    \begin{definition}
        Let $L$ be a vocabulary with approximations. We extend the notion of weak negation for all $L$-formulae as follows.
        \begin{itemize}
            \item $\weakneg\left( \bigwedge_{i\in I}\psi \right) = \bigvee_{i\in I}\weakneg(\psi)$,
            \item $\weakneg\left( \bigvee_{i\in I}\psi \right) = \bigwedge_{i\in I}\weakneg(\psi)$,
            \item $\weakneg\left( \forall v\, \psi \right) = \exists v\, \weakneg(\psi)$, and
            \item $\weakneg\left( \exists v\, \psi \right) = \forall v\, \weakneg(\psi)$.
        \end{itemize}
    \end{definition}

    \begin{remark}
        \label{Remark: Strong negation}
        Even though our logic $\Logic{\kappa}$ does not have negation, when $\kappa>\omega$, negation is definable in any class $\class$, such that $L$ has approximations for $\class$, via weak negation as follows: for atomic $\varphi$, by Definition~\ref{Definition: Vocabulary with approximations}~\ref{item: Weak negation}~\ref{item: Weak negation property}, $\A\not\models\varphi$ if and only if $\A\models\weakneg(\appr(\varphi,\varepsilon))$ for some $\varepsilon>0$, so $\neg\varphi \equiv \bigvee_{\varepsilon\in D\setminus\{0\}}\weakneg(\appr(\varphi,\varepsilon))$. Then the negations of more complicated formulae can be defined in negation normal form.
    \end{remark}

    \subsection{Isomorphism notions}\label{Subsection: Isomorphism notions}

    We study an abstract framework of $\varepsilon$-isomorphisms, for $\varepsilon\geq 0$, on a class $\class$ of metric structures, partially motivated by a pseudometric on $\class$ that arises from this notion of $\varepsilon$-isomorphism. In our examples (see Section~\ref{Section: Examples}), this pseudometric corresponds to well-known distances found in literature.

    \begin{definition}
        \label{Definition: Isomorphism notion}
        Let $L$ be a vocabulary with approximations with respect to a class $\class$ of $L$-structures, and, for any $\varepsilon\geq 0$, let $\isom_\varepsilon$ be a class of bijections $f\colon\A\to\B$, for $\A,\B\in\class$, such that the following hold.
        \begin{enumerate}

            \item If $f\colon\A\to\B$ is an $L$-isomorphism between $\A$ and $\B$, for $\A,\B\in\class$, then $f\in\isom_0$.

            \item If $f\in\isom_\varepsilon$, then $f^{-1}\in\isom_\varepsilon$.

            \item If $f\in\isom_\varepsilon$ and $g\in\isom_{\varepsilon'}$ and $\ran(f) = \dom(g)$, then $g\circ f\in\isom_{\varepsilon+\varepsilon'}$.

            \item If $\varepsilon\leq\varepsilon'$, then $\isom_{\varepsilon}\subseteq\isom_{\varepsilon'}$.

            \item\label{item: Preserving atomic formulae suffices for epsilon-isomorphism}
            For $\A,\B\in\class$, $\varepsilon\geq 0$ and a surjection $f\colon\A\to\B$, we have $f\in\isom_\varepsilon$ if and only if for all atomic $\varphi(\bar{v})$ and $\bar{a}\in\A^{\abs{\bar{v}}}$,
            \[
                \A\models\varphi(\bar{a}) \implies \B\models\appr(\varphi,\varepsilon)(f(\bar{a})).
            \]

        \end{enumerate}
        Then we call $\isom=(\isom_\varepsilon)_{\varepsilon\geq 0}$ an isomorphism notion of $\class$ and elements of $\isom_\varepsilon$ $\varepsilon$-iso\-morph\-isms.
    \end{definition}

    \begin{lemma}
        If $f\colon\A\to\B$ is an $\varepsilon$-isomorphism, then for all $\varphi(\bar{v})\in\Logic{\kappa}[L]$ and $\bar{a}\in\A^{\abs{\bar{v}}}$, we have
        \[
            \A\models\varphi(\bar{a}) \implies \B\models\appr(\varphi,\varepsilon)(f(\bar{a}))
        \]
        and
        \[
            \B\models\varphi(f(\bar{a})) \implies \A\models\appr(\varphi,\varepsilon)(\bar{a}).
        \]
    \end{lemma}
    \begin{proof}
        Follows immediately from the definitions.
    \end{proof}

    \begin{definition}
        Given a class $\class$ and an isomorphism notion $\isom=(\isom_\varepsilon)_{\varepsilon\geq 0}$, we can define a pseudometric $d_{\isom}$ on $\class$ by setting
        \[
            d_{\isom}(\A,\B) \coloneqq \inf\{\varepsilon\geq 0 \mid \text{there exists an $\varepsilon$-isomorphism $\A\to\B$}\}.
        \]
    \end{definition}

    \begin{definition}
        \label{Definition: Good vocabulary (isomorphism)}
        Let $L$ be a vocabulary with approximations with respect to a class $\class$ of $L$-structures and $\isom=(\isom_\varepsilon)_{\varepsilon\geq 0}$ an isomorphism notion of $\class$. Denote by $\Phi$ the set of atomic $L$-formulae.
        If there are
        sets $\good(k)\subseteq\Phi$, $k<\omega$,
        such that the following conditions hold, then we say that $L$ \emph{is good for} $(\class,\isom)$.
        \begin{enumerate}

            \item\label{item: Every atomic formula is k-good for some k}
            For any $\varphi\in\Phi$ there are $k<\omega$ and $\psi\in\good(k)$ such that $\varphi\equiv\psi$.

            \item\label{item: k-good formulae increase when k increases}
            $\good(k)\subseteq\good(k+1)$ for all $k<\omega$.

            \item\label{item: An approximating of a k-good formula is still k-good for a little big bigger k}
            For $k<\omega$ and $\varepsilon\geq 0$, if $\varphi\in\good(k)$, then $\appr(\varphi,\varepsilon)\in\good(\ceil{ke^{\varepsilon}})$ and $\weakneg(\appr(\varphi,\varepsilon))\in\good(\ceil{ke^{\varepsilon}})$.

            \item\label{item: Perturbation distance}
            For any $\A\in\class$, $k<\omega$ and $\varepsilon>0$ there exists $\delta>0$ such that for any $\varphi(\bar{v})\in\good(k)$ and $\bar{a},\bar{b}\in\A^{\abs{\bar{v}}}$, if $d(\bar{a},\bar{b})<\delta$, then
            \[
                \A\models\varphi(\bar{a}) \implies \A\models\appr(\varphi,\varepsilon)(\bar{b}).
            \]

            \item\label{item: Simultaneous Cauchy}
            Let $\A,\B\in\class$ and $\varepsilon\geq 0$. Let $(a_n)_{n<\omega}\in\A^\omega$ and $(b_n)_{n<\omega}\in\B^\omega$, and suppose that for some sequence $(\varepsilon_n)_{n<\omega}\in(\varepsilon,\infty)^\omega$ converging down to $\varepsilon$, for all $k<\omega$ and for all $\varphi(v_0,\dots,v_{m-1})\in\good(k)$ we have
            \[
                \A\models\varphi(a_{i_0},\dots,a_{i_{m-1}}) \implies \B\models\appr(\varphi,\varepsilon_k)(b_{i_0},\dots,b_{i_{m-1}})
            \]
            whenever $i_0,\dots,i_{m-1}\geq k$. Then, for any increasing sequence $(i_n)_{n<\omega}$ of indices, $(a_{i_n})_{n<\omega}$ is Cauchy\footnote{Cauchy being in the sense of the definable metric of Definition~\ref{Definition: Vocabulary with approximations}~\ref{item: Definable metric}.} if and only if $(b_{i_n})_{n<\omega}$ is Cauchy.

        \end{enumerate}
        We call elements of $\good(k)$ \emph{$k$-good formulae}.
        We call the number $\delta$ in item~\ref{item: Perturbation distance} a \emph{perturbation distance} for $k$ and $\varepsilon$ in $\A$.
    \end{definition}

    \begin{lemma}
        \label{Lemma: Limits satisfy k-good formulae}
        For sequences $(a_n)_{n<\omega}\in\A^\omega$ and $(b_n)_{n<\omega}\in\B^\omega$, and a sequence $(\varepsilon_n)_{n<\omega}\in(\varepsilon,\infty)^\omega$ converging down to $\varepsilon>0$, if for all $k<\omega$ and for all $\varphi(v_0,\dots,v_{m-1})\in\good(k)$ we have
        \[
            \A\models\varphi(a_{i_0},\dots,a_{i_{m-1}}) \implies \B\models\appr(\varphi,{\varepsilon_k})(b_{i_0},\dots,b_{i_{m-1}})
        \]
        whenever $i_0,\dots,i_{m-1}\geq k$, then the following holds: for any atomic $\varphi(v_0,\dots,v_{m-1})$ and increasing sequences $(i^j_n)_{n<\omega}\in\omega^\omega$, $j<m$, $(a_{i^j_n})_{n<\omega}$ converges if and only if $(b_{i^j_n})_{n<\omega}$ converges for all $j<m$, and when they do, we have
        \[
            \A\models\varphi(a^0,\dots,a^{m-1}) \implies \B\models\appr(\varphi,\varepsilon)(b^0,\dots,b^{m-1}),
        \]
        where $a^j = \lim_{n\to\infty}a^j_{i_n}$ and $b^j = \lim_{n\to\infty}b^j_{i_n}$, $j<m$.
    \end{lemma}
    \begin{proof}
        Let $\varphi(v_0,\dots,v_{m-1})$ be an atomic formula, and let$(i^j_n)_{n<\omega}$, $j<m$, be increasing sequences. By Definition~\ref{Definition: Good vocabulary (isomorphism)}~\ref{item: Simultaneous Cauchy}, $(a_{i^j_n})_{n<\omega}$ converges if and only if $(b_{i^j_n})_{n<\omega}$ converges. So suppose that they do and denote by $a^{j}$ and $b^{j}$ the limits, for $j<m$. Suppose that
        \begin{equation}
            \A\models\varphi(a^0,\dots,a^{m-1}).
            \label{eq: Convergence lemma eq 1}
        \end{equation}
        Let $k<\omega$ be large enough that $\appr(\varphi,{1/k+\varepsilon_k})\in\good(k)$. Let $\delta$ be a perturbation distance for $k$ and $1/k$ in both $\A$ and $\B$. Let $n_k\geq k$ be so large that for $n>n_k$, we have $d(a_{i^j_n},a^j),d(b_{i^j_n},b^j)<\delta$ for all $j<m$.

        Fix $n>n_k$. Then, as $\varphi\in\good(k)$ and $d(a_{i^j_n},a^j)<\delta$ for $j<m$, from~\eqref{eq: Convergence lemma eq 1} and Definition~\ref{Definition: Good vocabulary (isomorphism)}~\ref{item: Perturbation distance}, it follows that
        \[
            \A\models\appr(\varphi,{1/k})(a_{i^0_n},\dots,a_{i^{m-1}_n}).
        \]
        As $i^j_n>n_k\geq k$ for $j<m$ and $\appr(\varphi,{1/k})\in\good(k)$, by the initial assumption we get
        \begin{equation}
            \B\models\appr((\varphi,{1/k}),{\varepsilon_k})(b_{i^0_n},\dots,b_{i^{m-1}_n}).
            \label{eq: Convergence lemma eq 2}
        \end{equation}
        Now $\appr((\varphi,{1/k}),{\varepsilon_k}) = \appr(\varphi,{1/k+\varepsilon_k})$. As $\appr(\varphi,{1/k+\varepsilon_k})\in\good(k)$ and $d(b_{i^j_n},b^j)<\delta$ for $j<m$, we now have, by~\eqref{eq: Convergence lemma eq 2} and Definition~\ref{Definition: Good vocabulary (isomorphism)}~\ref{item: Perturbation distance},
        \[
            \B\models\appr((\varphi,{1/k+\varepsilon_k}),{1/k})(b^0,\dots,b^{m-1}).
        \]

        Now $\appr(\appr(\varphi,{1/k+\varepsilon_k}),{1/k}) = \appr(\varphi,{2/k+\varepsilon_k})$, so we have shown that for arbitrarily large $k<\omega$,
        \[
            \B\models\appr(\varphi,{2/k+\varepsilon_k})(b^0,\dots,b^{m-1}).
        \]
        As $2/k+\varepsilon_k\to\varepsilon$ when $k\to\infty$, this means that $\B\models\appr(\appr(\varphi,\varepsilon),\zeta)(b^0,\dots,b^{m-1})$ for all $\zeta>0$, whence, by Definition~\ref{Definition: Vocabulary with approximations}~\ref{item: Approximation}~\ref{item: An atomic formula is satisfied iff each of its approximations is satisfied},
        \[
            \B\models\appr(\varphi,\varepsilon)(b^0,\dots,b^{m-1}). \qedhere
        \]
    \end{proof}

    \begin{lemma}
        \label{Lemma: Perturbation distance function}
        Let $L$ be good for $(\class,\isom)$. Then for any $\A\in\class$, there is a function $\delta\colon(0,\infty)\times\omega\to(0,\infty)$ that is increasing in the first argument and decreasing in the second argument, such that for any $\varepsilon>0$ and $k<\omega$, the number $\delta(\varepsilon,k)$ is a perturbation distance for $\varepsilon$ and $k$ in $\A$.
    \end{lemma}
    \begin{proof}
        First, for $k<\omega$, let $\delta_k\colon(0,\infty)\to(0,\infty)$ be the function defined by
        \[
            \delta_k(\varepsilon) = \sup\{r \mid \text{r is a perturbation distance for $\varepsilon$ and $k$}\}.
        \]
        Then $\delta_k(\varepsilon)$ is a perturbation distance for $k$ and $\varepsilon$ and increasing in $\varepsilon$.

        Now let $\delta(\varepsilon,k)=\min_{i\leq k}\delta_i(\varepsilon)$. Then clearly $\delta$ is decreasing in the second argument, and as each $\delta_i$ is increasing, $\delta$ is increasing in the first argument. Clearly $\delta(\varepsilon,k)$ is also a perturbation distance for $k$ and $\varepsilon$.
    \end{proof}

    \subsection{Correspondence notions}\label{Subsection: Correspondence notions}

    \begin{definition}
        By a \emph{(set) correspondence} between two sets $A$ and $B$, we mean simply a binary relation $R$ such that $\dom(R)=A$ and $\ran(R)=B$.
    \end{definition}

    \begin{definition}
        \label{Definition: Correspondence notion}
        Let $L$ be a vocabulary with approximations with respect to a class $\class$ of $L$-structures, and, for any $\varepsilon\geq 0$, let $\corr_\varepsilon$ be set of relations $R\subseteq\A\times\B$, for $\A,\B\in\class$, such that the following hold.
        \begin{enumerate}

            \item If $R\in\corr_\varepsilon$, then there are $\A,\B\in\class$ such that $R$ is a set correspondence between a dense\footnote{Dense being in the sense of the definable metric of Definition~\ref{Definition: Vocabulary with approximations}~\ref{item: Definable metric}.} subset of $\A$ and a dense subset of $\B$.

            \item If $f\colon\A\to\B$ is an $L$-isomorphism, for $\A,\B\in\class$, then $f\in\corr_0$.

            \item $R\in\corr_\varepsilon$ if and only if $R^{-1}\in\corr_\varepsilon$.

            \item If $\A,\B,\model C\in\class$, $R\in\corr_\varepsilon$ and $R'\in\corr_{\varepsilon'}$, $\dom(R)$ is dense in $\A$, $\ran(R)$ and $\dom(R')$ are dense in $\B$ and $\ran(R')$ is dense in $\model C$, then there is $R''\in \corr_{\varepsilon+\varepsilon'}$ such that $\dom(R'')$ is dense in $\A$ and $\ran(R'')$ is dense in $\model C$.

            \item\label{item: Preserving atomic formulae suffices for epsilon-correspondence}
            For $\A,\B\in\class$, $\varepsilon>0$ and any dense sets $\dense\A\subseteq\A$ and $\dense\B\subseteq\B$, the following are equivalent.
            \begin{enumerate}
                \item There is $\varepsilon'\in(0,\varepsilon)$ and $R\in\corr_{\varepsilon'}$ such that $\dom(R)=\dense\A$ and $\ran(R)=\dense\B$.
                \item There is $\varepsilon'\in(0,\varepsilon)$ and a set correspondence $R$ between $\dense\A$ and $\dense\B$ such that for all atomic $\varphi(v_0,\dots,v_{m-1})$ and $(a_i,b_i)\in R$, $i<m$, we have
                \[
                    \A\models\varphi(\bar{a}) \implies \B\models\appr(\varphi,\varepsilon')(\bar{b}).
                \]
            \end{enumerate}

        \end{enumerate}
        Then we call the sequence $\corr = (\corr_\varepsilon)_{\varepsilon\geq 0}$ a correspondence notion of $\class$ and elements of $\corr_\varepsilon$ $\varepsilon$-correspondences. If $R\in\corr_\varepsilon$, $\dom(R)$ is dense in $\A$ and $\ran(\B)$ is dense in $\B$, we then say that $R$ is an $\varepsilon$-correspondence \emph{between} $\A$ and $\B$.
    \end{definition}

    \begin{definition}
        Given a class $\class$ and a correspondence notion $\corr=(\corr_\varepsilon)_{\varepsilon\geq 0}$, we can define a pseudometric $d_{\corr}$ on $\class$ by setting
        \[
            d_{\corr}(\A,\B) \coloneqq \inf\{\varepsilon\geq 0 \mid \text{there exists an $\varepsilon$-correspondence between $\A$ and $\B$}\}.
        \]
    \end{definition}

    \section{Examples}\label{Section: Examples}

    We give examples of classes $\class$, isomorphism notions $\isom=(\isom_\varepsilon)_{\varepsilon\geq 0}$ and correspondence notions $\corr=(\corr_\varepsilon)_{\varepsilon\geq 0}$, and vocabularies $L$ such that $L$ is good for $(\class,\isom)$.

    We consider three classes of structures:
    \begin{itemize}
        \item the class $\MS$ of all complete metric spaces,
        \item the class $\BS$ of all real\footnote{For simplicity of notation we only consider real Banach spaces, but conceptually complex Banach spaces are no different and generalizing the theory to cover complex spaces is trivial.} Banach spaces and
        \item the class $\BSb$ of the closed unit balls of all real Banach spaces.
    \end{itemize}
    We denote
    \begin{itemize}
        \item $\LM = \{\dr_r,\Dr_r \mid r\in D\cap[0,\infty) \}$,
        \item $\LB = \{+,f_\lambda, P, Q \mid \lambda\in D\}$ and
        \item $\LBb = \{P_{r}, Q_{r}, g_{\lambda_0,\dots,\lambda_{n-1}} \mid r\in D\cap[0,1] , n < \omega, \lambda_i\in D, \sum_{i<n}\abs{\lambda_i}=1 \}$,
    \end{itemize}
    where $\dr_r$ and $\Dr_r$ are binary relation symbols, $+$ is a binary function symbol, $f_\lambda$ is a unary function symbol, $P$, $Q$, $P_r$ and $Q_r$ are unary relation symbols, $g_{\lambda_0,\dots,\lambda_{n-1}}$ is an $n$-ary function symbol and $D$ is the closure of $\Q$ under addition, multiplication, division by non-zero elements and the exponential function $e^x$.

    We consider $\MS$ as a class of $\LM$-structures by interpreting the atomic formulae as follows:
    \[
        \A\models\dr_r(a,b) \iff d_\A(a,b)\leq r,
    \]
    and
    \[
        \A\models\Dr_r(a,b) \iff d_\A(a,b)\geq r.
    \]
    We consider $\BS$ as a class of $\LB$-structures by letting $+^{\A}$ be the addition of the Banach space $\A$, $f_\lambda^{\A}$ the scalar multiplication by the real $\lambda$ and
    \[
        \A\models P(t(\bar{a})) \iff \norm{t^\A(\bar{a})} \leq 1
    \]
    and
    \[
        \A\models Q(t(\bar{a})) \iff \norm{t^\A(\bar{a})} \geq 1
    \]
    for all $\LB$-terms $t(\bar{v})$. We usually denote the term $f_\lambda(t)$ just by $\lambda t$. Similarly, we consider $\BSb$ as a class of $\LBb$-structures by letting $g_{\lambda_0,\dots,\lambda_{n-1}}^\A(a_0,\dots,a_{n-1}) = \sum_{i<n}\lambda_i a_i$ and
    \[
        \A\models P_r(t(\bar{a}) \iff \norm{t^\A(\bar{a})} \leq r
    \]
    and
    \[
        \A\models Q_r(t(\bar{a}) \iff \norm{t^\A(\bar{a})} \geq r
    \]
    for all $\LBb$-terms $t(\bar{v})$. We usually denote the term $g_{\lambda_0,\dots,\lambda_{n-1}}(t_0,\dots,t_{n-1})$ simply by $\sum_{i<n}\lambda_i t_i$.

    \subsection{Metric spaces and an isomorphism notion}

    Define an approximation $\appr$ of $\LM$ by setting $\appr(\dr_r(u,v),\varepsilon) = \dr_{e^\varepsilon r}(u,v)$ and $\appr(\Dr_r(u,v),\varepsilon) = \Dr_{e^{-\varepsilon}r}(u,v)$. Define $\weakneg$ by setting $\weakneg(\dr_r(u,v)) = \Dr_r(u,v)$ and $\weakneg(\Dr_r(u,v)) = \dr_r(u,v)$. It is easy to check that $\LM$ together with $\appr$ and $\weakneg$ is a vocabulary with approximations with respect to $\MS$.

    Let $\isom_\varepsilon$ be the class of all $e^\varepsilon$-bi-Lipschitz homeomorphisms $\A\to\B$ for $\A,\B\in\MS$. Then $\isom = (\isom_\varepsilon)_{\varepsilon\geq 0}$ is an isomorphism notion of $\MS$. We check condition~\ref{item: Preserving atomic formulae suffices for epsilon-isomorphism} of Definition~\ref{Definition: Isomorphism notion} in Lemmata~\ref{Lemma: Lipschitz isomorphism preserves formulae} and~\ref{Lemma: For Lipschitz isomorphism suffices to preserve k-good formulae}, other conditions are clear.

    Given $k<\omega$, let $\good(k)$ be the set of $\dr_r(u,v)$ and $\Dr_r(u,v)$ such that $r\geq 1/k$. We show that with these definitions, $\LM$ is good for $(\MS,\isom)$. In Lemmata~\ref{Lemma: Lipschitz perturbation distance} and~\ref{Lemma: Lipschitz tail-dense sequences behave nicely}, we check the conditions~\ref{item: Perturbation distance} and~\ref{item: Simultaneous Cauchy} of Definition~\ref{Definition: Good vocabulary (isomorphism)}; the rest are clear.

    \begin{lemma}
        \label{Lemma: Lipschitz vice versa follows if satisfaction is approximate}
        Let $k<\omega$, $\varepsilon\geq 0$, $\A,\B\in\MS$, $a,a'\in\A$ and $b,b'\in\B$. Suppose that for all $\varphi(v,u)\in\good(k)$,
        \[
            \A\models\varphi(a,a') \implies \B\models\appr(\varphi,\varepsilon)(b,b').
        \]
        Then for all $r\geq 1/k$, also
        \[
            \B\models \dr_r(b,b') \implies \A\models \appr(\dr_r(u,v),\varepsilon)(a,a').
        \]
    \end{lemma}
    \begin{proof}
        We show the contrapositive of the claim. Suppose that $\A\not\models \appr(\dr_r(u,v),\varepsilon)(a,a')$, i.e. $d(a,a')>e^\varepsilon r$. Then there is $\varepsilon'>0$ such that $d(a,a')\geq e^{\varepsilon'}e^\varepsilon r$. But then $\A\models \Dr_{e^{\varepsilon'+\varepsilon}r}(a,a')$, and $\Dr_{e^{\varepsilon'+\varepsilon}r}(u,v)$ is still in $\good(k)$, so by the initial assumption $\B\models \appr(\Dr_{e^{\varepsilon'+\varepsilon}r}(u,v),\varepsilon)(b,b')$, i.e. $d(b,b')\geq e^{\varepsilon'}r>r$. Thus $\B\not\models \dr_r(b,b')$.
    \end{proof}

    \begin{lemma}
        \label{Lemma: Lipschitz isomorphism preserves formulae}
        Let $f\colon\A\to\B$ be a surjection. If $f\in\isom_\varepsilon$, then for all atomic $\varphi(u,v)$ and $a,a'\in\A$,
        \[
            \A\models\varphi(a,a') \implies \B\models\appr(\varphi,{\varepsilon})(f(a),f(a')).
        \]
    \end{lemma}
    \begin{proof}
        Clear.
    \end{proof}

    \begin{lemma}
        \label{Lemma: For Lipschitz isomorphism suffices to preserve k-good formulae}
        Let $f\colon\A\to\B$ be a surjection. If for all atomic $\varphi(u,v)$ and $a,a'\in\A$,
        \[
            \A\models\varphi(a,a') \implies \B\models\appr(\varphi,{\varepsilon})(f(a),f(a')),
        \]
        then $f\in\isom_\varepsilon$.
    \end{lemma}
    \begin{proof}
        Let $a,a'\in\A$. Then for any $r>0$,
        \begin{align*}
            d(a,a') < r &\implies \A\models\dr_r(a,a') \\
            &\implies \B\models\dr_{e^\varepsilon r}(f(a),f(a')) \\
            &\implies e^{-\varepsilon}d(f(a),f(a'))\leq r,
        \end{align*}
        whence $e^{-\varepsilon}d(f(a),f(a'))\leq d(a,a')$, proving that $f$ is $e^\varepsilon$-Lipschitz. By Lemma~\ref{Lemma: Lipschitz vice versa follows if satisfaction is approximate}, noticing that $\dr_r$ is $k$-good for any $k\geq 1/r$, we can use a symmetric argument to show that $e^{-\varepsilon}d(a,a')\leq d(f(a),f(a'))$. Hence $f$ is $e^\varepsilon$-bi-Lipschitz.
    \end{proof}

    \begin{lemma}
        \label{Lemma: Lipschitz perturbation distance}
        For any $\A\in\class$, $k<\omega$ and $\varepsilon>0$, there exists $\delta>0$ such that for any $\varphi(u,v)\in\good(k)$ and $a,a',b,b'\in\A$, if $d(a,b),d(a',b')<\delta$, then
        \[
            \A\models\varphi(a,a')\implies\A\models\appr(\varphi,\varepsilon)(b,b').
        \]
    \end{lemma}
    \begin{proof}
        Let $k<\omega$ and $\varepsilon>0$ be arbitrary and let $\delta_{\dr} = (e^\varepsilon-1)/2k$ and $\delta_{\Dr} = (1-e^{-\varepsilon})/2k$. Let $r\geq 1/k$ and $a,a',b,b'\in\A^k$. Now it is easy to see that whenever $d(a,b),d(a',b')<\delta_{\dr}$, we have
        \[
            \A\models\dr_r(a,a') \implies \A\models\dr_{e^\varepsilon r}(b,b'),
        \]
        and whenever $d(a,b),d(a',b')<\delta_{\Dr}$, we have
        \[
            \A\not\models\Dr_{e^{-\varepsilon}r}(b,b') \implies \A\not\models\Dr_r(a,a').
        \]
        Thus $\delta \coloneqq \min\{\delta_{\dr},\delta_{\Dr}\}$ suffices.
    \end{proof}

    \begin{lemma}
        \label{Lemma: Lipschitz tail-dense sequences behave nicely}
        For sequences $(a_n)_{n<\omega}\in\A^\omega$ and $(b_n)_{n<\omega}\in\B^\omega$ and a sequence $(\varepsilon_n)_{n<\omega}$ of positive numbers converging down to $\varepsilon>0$, if for all $k<\omega$ and $\varphi(u,v)\in\good(k)$ we have
        \[
            \A\models\varphi(a_i, a_j) \implies \B\models\appr(\varphi,{\varepsilon_k})(b_i, b_j)
        \]
        for all  $i,j\geq k$, then, for any increasing sequence $(i_n)_{n<\omega}$ of indices, $(a_{i_n})_{n<\omega}$ is Cauchy if and only if $(b_{i_n})_{n<\omega}$ is Cauchy.
    \end{lemma}
    \begin{proof}
        Let $(i_n)_{n<\omega}$ be an increasing sequence of indices and suppose $(a_{i_n})_{n<\omega}$ is Cauchy. Let $r>0$. As $(a_{i_n})_{n<\omega}$ is Cauchy, there is $k<\omega$ such that $n,m\geq k$ implies $d(a_n,a_m)\leq e^{-\varepsilon-1}r$. One can always choose $k$ to be large enough that $e^{-\varepsilon-1}r\geq 1/k$ and $\varepsilon_k-\varepsilon\leq 1$. Then $\dr_{e^{-\varepsilon-1}r}(u,v)\in\good(k)$ and $\A\models\dr_{e^{-\varepsilon-1}r}(a_n,a_m)$ for any $n,m\geq k$, so $\B\models\appr(\dr_{e^{-\varepsilon-1}r}(u,v),{\varepsilon_k})(b_n,b_m)$, i.e.
        \[
            d(b_n,b_m)\leq e^{\varepsilon_{k}}e^{-\varepsilon-1}r = e^{\varepsilon_k-\varepsilon-1}r\leq e^0 r \leq r.
        \]
        Thus $(b_{i_n})_{n<\omega}$ is Cauchy. By Lemma~\ref{Lemma: Lipschitz vice versa follows if satisfaction is approximate}, the other direction can be proved by using a symmetric argument.
    \end{proof}

    \begin{definition}
        The \emph{Lipschitz distance} $\dl(X,Y)$ of two metric spaces $X$ and $Y$ is the infimum of all $\varepsilon\geq 0$ such that there exists an $e^\varepsilon$-bi-Lipschitz homeomorphism $f\colon X\to Y$.
    \end{definition}

    \begin{lemma}
        For $\A,\B\in\MS$, we have $\dl(\A,\B)=d_{\isom}(\A,\B)$.
    \end{lemma}
    \begin{proof}
        Check the definitions.
    \end{proof}

    \subsection{Metric spaces and a correspondence notion}

    Define another approximation $\appr$ of $\LM$ by setting $\appr(\dr_r(u,v),\varepsilon) = \dr_{r+2\varepsilon}(u,v)$ and $\appr(\Dr_r(u,v),\varepsilon) = \Dr_{r\dotminus 2\varepsilon}(u,v)$.\footnote{By $x\dotminus y$ we denote $\max\{x-y, 0\}$.} The function $\weakneg$ is defined as in the isomorphism notion case. Then $\LM$ with $\appr$ and $\weakneg$ is a vocabulary with approximations with respect to $\MS$.

    We let $\corr_\varepsilon$ be the class of all set correspondences $R$ between a dense set of $\A$ and dense set of $\B$, for $\A,\B\in\MS$, such that there exist isometric embeddings $\iota_{\A}\colon\dom(R)\to\model{C}$ and $\iota_{\B}\colon\ran(R)\to\model{C}$ into a third space $\model{C}\in\MS$ and for all $(a,b)\in R$, $d(\iota_{\A}(a),\iota_{\B}(b)) \leq \varepsilon$. Then it is not difficult to check that $\corr=(\corr_\varepsilon)_{\varepsilon\geq 0}$ is a correspondence notion of $\MS$. Only one item of Definition~\ref{Definition: Correspondence notion}, that is~\ref{item: Preserving atomic formulae suffices for epsilon-correspondence}, is nontrivial and is proven in Corollary~\ref{Corollary: Gromov-hausdorff formulae vs correspondences}.

    \begin{lemma}
        \label{Lemma: Gromov-Hausdorff Correspondence preserves truth}
        Let $\A,\B\in\MS$, and let $R\subseteq\A\times\B$. If $R\in\corr_\varepsilon$, then, whenever $(a,b),(a',b')\in R$ and $\varphi(u,v)$ is an atomic formula, we have
        \[
            \A\models\varphi(a,a') \implies \B\models\appr(\varphi,{\varepsilon})(b,b').
        \]
    \end{lemma}
    \begin{proof}
        Clear.
    \end{proof}

    \begin{lemma}
        \label{Lemma: Gromov-Hausdorff Preserving truth gives correspondence}
        Let $\A,\B\in\MS$. Let $R\subseteq\A\times\B$ be a set correspondence between a dense set of $\A$ and a dense set of $\B$, and suppose that whenever $(a,b),(a',b')\in R$, we have
        \[
            \A\models\varphi(a,a') \implies \B\models\appr(\varphi,{\varepsilon})(b,b')
        \]
        for all atomic formulae $\varphi(u,v)$. Then $R\in\corr_\varepsilon$.
    \end{lemma}
    \begin{proof}
        The main idea of this proof is the same as in e.g. the proof of Theorem~7.3.25 in~\cite{BBI01}.
    
        Let $\model{C}$ be the disjoint union of $\A$ and $\B$, and define a function $d_{\model C}\colon\model{C}\to[0,\infty)$ by setting $d_{\model C}(a,a') = d_{\A}(a,a')$ for $a,a'\in\A$, $d_{\model C}(b,b')=d_{\B}(b,b')$ for $b,b'\in\B$ and
        \[
            d_{\model C}(a,b) \coloneqq \inf\{d(a,a') + d(b,b') + \varepsilon \mid (a',b')\in R \}
        \]
        for $a\in\A$ and $b\in\B$ (and $d_{\model C}(b,a)$ similarly). We show that $d_{\model C}$ is a pseudometric. The only nontrivial requirement is the triangle inequality. For that, suffices to show that for all $a,a'\in\A$ and $b,b'\in\B$ we have
        \begin{align*}
            d_{\model C}(a,a') &\leq d_{\model C}(a,b)  + d_{\model C}(b,a'), \\
            d_{\model C}(b,b') &\leq d_{\model C}(b,a)  + d_{\model C}(a,b'), \\
            d_{\model C}(a,b)  &\leq d_{\model C}(a,a') + d_{\model C}(a',b) \text{ and} \\
            d_{\model C}(a,b)  &\leq d_{\model C}(a,b') + d_{\model C}(b',b).
        \end{align*}
        We only show that $d_{\model C}(a,a') \leq d_{\model C}(a,b) + d_{\model C}(b,a')$, the others are similar.

        Let $r>0$ and suppose that $d_{\model C}(a,b)  + d_{\model C}(b,a')<r$. Then there are $(a'',b''),(a''',b''')\in R$ with
        \[
            d(a,a'')+d(b,b'')+\varepsilon + d(a',a''')+d(b,b''')+\varepsilon < r.
        \]
        Noting that $d(b'',b''')\leq d(b,b'')+d(b,b''')$, we get that
        \[
            d(b'',b''')<(r - d(a,a'') - d(a',a'''))-2\varepsilon,
        \]
        i.e. $\B\not\models\appr(\Dr_{r - d(a,a'') - d(a',a''')}(u,v),\varepsilon)(b'',b''')$. Thus $\A\not\models\Dr_{r - d(a,a'') - d(a',a''')}(a'',a''')$, i.e. $d(a,a'') + d(a'',a''') + d(a''',a') < r$. But then
        \[
            d_{\model C}(a,a') \leq d(a,a'') + d(a'',a''') + d(a''',a') < r.
        \]
        As $r$ was arbitrary, we obtain $d_{\model C}(a,a') \leq d_{\model C}(a,b) + d_{\model C}(b,a')$.

        Now the inclusion mappings $\A\to\model C$ and $\B\to\model C$ are isometric embeddings. Left is to show that for any $(a,b)\in R$ we have $d_{\model C}(a,b)\leq\varepsilon$. Suppose that $(a,b)\in R$. Then by definition,
        \[
            d_{\model C}(a,b) \leq d(a,a) + d(b,b) + \varepsilon = \varepsilon.
        \]
        This verifies that $R\in\corr_\varepsilon$.
    \end{proof}

    \begin{corollary}
        \label{Corollary: Gromov-hausdorff formulae vs correspondences}
        Let $\A,\B\in\MS$ and let $\dense\A\subseteq\A$ and $\dense\B\subseteq\B$ be dense sets of the respective spaces. Then for any $\varepsilon>0$ the following are equivalent.
        \begin{enumerate}
            \item There is $\varepsilon'\in(0,\varepsilon)$ and $R\in\corr_{\varepsilon'}$ such that $\dom(R)=\dense\A$ and $\ran(R)=\dense\B$.
            \item There is $\varepsilon'\in(0,\varepsilon)$ and a set correspondence $R$ between $\dense\A$ and $\dense\B$ such that whenever $(a,b),(a',b')\in R$, $i<n$, and $\varphi(v,u)$ is an atomic formula, we have
            \[
                \A\models\varphi(a,a') \implies \B\models\appr(\varphi,\varepsilon')(b,b').
            \]
        \end{enumerate}
    \end{corollary}
    \begin{proof}
        Putting together Lemmata~\ref{Lemma: Gromov-Hausdorff Correspondence preserves truth} and~\ref{Lemma: Gromov-Hausdorff Preserving truth gives correspondence}, we obtain the following: for any set correspondence $R$ between a dense set of $\A$ and a dense set of $\B$, $R\in\corr_\varepsilon$ if and only if
        \[
            \A\models\varphi(a,a') \implies \B\models\appr(\varphi,\varepsilon')(b,b')
        \]
        holds for all $(a,b),(a',b')\in R$ and atomic formulae $\varphi(v,u)$. This stronger than our claim.
    \end{proof}

    \begin{definition}
        The \emph{Hausdorff distance} $\dH{X}(A,B)$ of two sets $A$ and $B$ in a metric space $X$ is the maximum of the two numbers
        \[
            \sup_{a\in A}d(a,B) \quad\text{and}\quad \sup_{b\in B}d(b,A).
        \]
        The \emph{Gromov--Hausdorff distance} $\dgh(X,Y)$ of two metric spaces $X$ and $Y$ is the infimum of
        \[
            \dH{Z}(\iota_{X}(X),\iota_{Y}(Y))
        \]
        over all isometric embeddings $\iota_{X}\colon X\to Z$ and $\iota_{Y}\colon Y\to Z$ into a third metric space $Z$.
    \end{definition}

    \begin{lemma}
        \label{Lemma: GH-distance can be evaluated in dense sets}
        Let $\A$ and $\B$ be metric spaces and $\dense\A$ and $\dense\B$ dense sets of the respective spaces. Then $\dgh(\dense\A,\dense\B) = \dgh(\A,\B)$.
    \end{lemma}
    \begin{proof}
        An easy exercise in metric topology.
    \end{proof}

    \begin{lemma}
        \label{Lemma: Gromov-Hausdorff distance is captured by correspondence notion}
        For $\A,\B\in\MS$, we have $\dgh(\A,\B) = d_\corr(\A,\B)$.
    \end{lemma}
    \begin{proof}
        By Lemma~\ref{Lemma: GH-distance can be evaluated in dense sets}, it is enough to show that for  for any $\varepsilon>0$,
        \begin{enumerate}
            \item if $\dgh(\A,\B)<\varepsilon$, then $d_\corr(\A,\B)\leq\varepsilon$, and
            \item if $d_\corr(\A,\B)<\varepsilon$, then $\dgh(\dense\A,\dense\B)\leq\varepsilon$ for some dense $\dense\A\subseteq\A$ and $\dense\B\subseteq\B$.
        \end{enumerate}

        First suppose that $\dgh(\A,\B)<\varepsilon$. Then by definition, there are isometric embeddings $\iota_{\A}\colon\A\to\model{C}$ and $\iota_{\B}\colon\B\to\model{C}$ into a third space $\model{C}$ such that
        \[
            \dH{\model C}(\iota_{\A}(\A),\iota_{\B}(\B)) < \varepsilon.
        \]
        Let $R$ be the set of all $(a,b)\in\A\times\B$ such that $d(\iota_{\A}(a),\iota_{\B}(b))\leq\varepsilon$. Now if $\dom(R)=\A$ and $\ran(R)=\B$, then $R$ is clearly an $\varepsilon$-correspondence. So let $a\in\A$. Then
        \begin{align*}
            d(\iota_{\A}(a),\iota_\B(\B)) &\leq \sup_{a'\in\A}d(\iota_{\A}(a'),\iota_\B(\B)) \leq \dH{\model C}(\iota_\A(\A),\iota_\B(\B)) < \varepsilon.
        \end{align*}
        As $d(\iota_{\A}(a),\iota_\B(\B))=\inf_{b\in\B}d(\iota_{\A}(a),\iota_{\B}(b))$, this means that there is some $b\in\B$ with $d(\iota_{\A}(a),\iota_{\B}(b))<\varepsilon$, whence $(a,b)\in R$. Thus $\dom(R)=\A$. Similarly we get that $\ran(R)=\B$. As $R$ is an $\varepsilon$-correspondence, it follows that $d_\corr(\A,\B)\leq\varepsilon$.

        Then suppose that $d_\corr(\A,\B)<\varepsilon$. Then there is an $\varepsilon$-correspondence $R$ between $\A$ and $\B$. This means that there are isometric embeddings $\iota_{\A}\colon\dom(R)\to\model{C}$ and $\iota_{\B}\colon\ran(R)\to\model{C}$ such that for all $(a,b)\in R$ we have $d(\iota_{\A}(a),\iota_{\B}(b)) \leq \varepsilon$. Now let $a\in\A$. Let $b\in\ran(b)$ be such that $(a,b)\in R$. Now
        \[
             d(a,\ran(R)) \leq d(a,b) \leq \varepsilon,
        \]
        whence $\sup_{a\in\dom(R)}d(a,\ran(R))\leq\varepsilon$. Similarly $\sup_{b\in\ran(R)}d(\dom(R),b)\leq$, so $\dH{\model C}(\dom(R),\ran(R))\leq\varepsilon$. Thus we have $\dgh(\dom(R),\ran(R))\leq\varepsilon$. As $R$ is an $\varepsilon$-correspondence, $\dom(R)$ is dense in $\A$ and $\ran(R)$ dense in $\B$, so we are done.
    \end{proof}

    \subsection{Banach spaces and an isomorphism notion}

    Define an approximation $\appr$ of $\LB$ by setting $\appr(P(t),\varepsilon)=P(e^{-\varepsilon}t)$ and $\appr(Q(t),\varepsilon)=Q(e^\varepsilon t)$. Define $\weakneg$ by setting $\weakneg(P(t)) = Q(t)$ and $\weakneg(Q(t)) = P(t)$. It is easy to check that $\LB$ together with $\appr$ and $\weakneg$ is a vocabulary with approximations with respect to $\BS$.

    Let $\isom_\varepsilon$ be the class of all linear $e^\varepsilon$-bi-Lipschitz homeomorphisms $\A\to\B$ for $\A,\B\in\BS$. Then $\isom = (\isom_\varepsilon)_{\varepsilon\geq 0}$ is an isomorphism notion of $\class$. We check condition~\ref{item: Preserving atomic formulae suffices for epsilon-isomorphism} of Definition~\ref{Definition: Isomorphism notion} in Lemmata~\ref{Lemma: Banach-Mazur isomorphism preserves formulae} and~\ref{Lemma: For Banach-Mazur isomorphism suffices to preserve k-good formulae}.

    Given $k<\omega$, let $\good(k)$ be the set of $P(t)$ and $Q(t)$ such that $t$ is a term of the form $\sum_{i=0}^{k-1} c_i v_i$, where $\abs{c_i} \leq k$ and each $v_i$ is a distinct variable. We call such terms $k$-good.\footnote{If if $0\leq m<k$ and $\abs{c_i}\leq k$ for $i<m$, then by $\sum_{i=0}^{m-1}c_i v_i$ we mean the $k$-good term $\sum_{i=0}^{k-1}c'_i v_i$, where $c'_i=c_i$ for $i<m$ and $c'_i=0$ for $m\leq i<k$. If $r\in [0,1]$ and $t=\sum_{i=0}^{k-1}c_i v_i$ is a $k$-good term, then by $rt$ we mean the $k$-good term $\sum_{i=0}rc_i v_i$.} We show that with these definitions, $\LB$ is good for $(\BS,\isom)$. In Lemmata~\ref{Lemma: Banach-Mazur perturbation distance} and~\ref{Lemma: Banach-Mazur tail-dense sequences behave nicely}, we check conditions~\ref{item: Perturbation distance} and~\ref{item: Simultaneous Cauchy} of Definition~\ref{Definition: Good vocabulary (isomorphism)}; the rest are clear.

    \begin{lemma}
        \label{Lemma: Banach-Mazur vice versa follows if satisfaction is approximate}
        Let $k<\omega$, $\varepsilon\geq 0$, $\A,\B\in\BS$, $\bar{a}\in\A^k$ and $\bar{b}\in\B^k$. Suppose that for all $k$-good term $t(\bar{v})$,
        \[
            \A\models Q(t(\bar{a})) \implies \B\models\appr(Q(t(\bar{v})),\varepsilon)(\bar{b}).
        \]
        Then, for all $k$-good terms $t(\bar{v})$,
        \[
            \B\models P(t(\bar{b})) \implies \A\models\appr(P(t(\bar{v})),\varepsilon)(\bar{a}).
        \]
    \end{lemma}
    \begin{proof}
        Let $t(\bar{v})$ be $k$-good and suppose that $\A\not\models\appr(P(t(\bar{v})),\varepsilon)(\bar{a})$, i.e. $\norm{e^{-\varepsilon} t^\A(\bar{a})}>1$. Then there is $r>1$ such that $\norm{e^{-\varepsilon}t^\A(\bar{a})}\geq r$. But then $\A\models Q(r^{-1}e^{-\varepsilon}t(\bar{a}))$, and $r^{-1}e^{-\varepsilon}t$ is still a $k$-good term, so by the initial assumption $\B\models \appr(Q(r^{-1}e^{-\varepsilon}t(\bar{v})),\varepsilon)(\bar{b})$, i.e. $\norm{t^\B(\bar{b})} \geq r > 1$. Thus $\B\not\models P(t(\bar{b}))$.
    \end{proof}

    \begin{lemma}
        \label{Lemma: Banach-Mazur isomorphism preserves formulae}
        Let $f\colon\A\to\B$ be a surjection. If $f\in\isom_\varepsilon$, then for all atomic $\varphi(\bar{v})$ and $\bar{a}\in\A^k$,
        \[
            \A\models\varphi(\bar{a}) \implies \B\models\appr(\varphi,{\varepsilon})(f(\bar{a}))
        \]
        and
        \[
            \B\models\varphi(f(\bar{a})) \implies \A\models\appr(\varphi,{\varepsilon})(\bar{a}).
        \]
    \end{lemma}
    \begin{proof}
        Let $\bar{a}\in\A^k$. We show that
        \[
            \A\models\varphi(\bar{a}) \implies \B\models\appr(\varphi,{\varepsilon})(f(\bar{a})),
        \]
        proof for the other claim is similar.

        First suppose that $\varphi = P(\sum_{i=0}^{k-1}c_iv_i)$. Now if $\A\models\varphi(\bar{a})$, $\norm{\sum_{i=0}^{k-1}c_ia_i}\leq 1$, so
        \[
            \norm{\sum_{i=0}^{k-1}c_if(a_i)} = \norm{f(\sum_{i=0}^{k-1}c_ia_i)} \leq e^\varepsilon\norm{\sum_{i=0}^{k-1}c_ia_i}\leq e^\varepsilon.
        \]
        Hence $\B\models P(e^{-\varepsilon}\sum_{i=0}^{k-1}c_if(a_i))$, so $\B\models\appr(\varphi,\varepsilon)(f(\bar{a}))$.

        The case for $\varphi=Q(\sum_{i=0}^kc_iv_i)$ is similar.
    \end{proof}

    \begin{lemma}
        \label{Lemma: For Banach-Mazur isomorphism suffices to preserve k-good formulae}
        Let $f\colon\A\to\B$ be a surjection. If for all atomic $\varphi(\bar{v})$ and $\bar{a}\in\A^k$,
        \[
            \A\models\varphi(\bar{a}) \implies \B\models\appr(\varphi,{\varepsilon})(f(\bar{a})),
        \]
        then $f\in\isom_\varepsilon$
    \end{lemma}
    \begin{proof}
        Let $a,a'\in\A$. Now we have for all $r>0$
        \begin{align*}
            \norm{a-a'} < r &\implies \A\models P(a/r - a'/r) \\
            &\implies \B\models P(e^{-\varepsilon}(f(a)/r-f(a')/r)) \\
            &\implies e^{-\varepsilon}\norm{f(a)-f(a')}\leq r,
        \end{align*}
        which shows that $e^{-\varepsilon}\norm{f(a)-f(a')}\leq\norm{a-a'}$, proving that $f$ is $e^\varepsilon$-Lipschitz. By Lemma~\ref{Lemma: Banach-Mazur vice versa follows if satisfaction is approximate}, we can use a symmetric argument to show that $e^{-\varepsilon}\norm{a-a'}\leq\norm{f(a)-f(a')}$. Hence $f$ is $e^\varepsilon$-bi-Lipschitz.

        Left is to show linearity. Let $a_0,a_1\in\A$ and $\lambda_0,\lambda_1\in D$. Denote by $a_2$ the element $\lambda_0 a_0 + \lambda_1 a_1$. Let $b_i=f(a_i)$ for $i=0,1,2$. Now, as $a_2 = \lambda_0 a_0 + \lambda_1 a_1$, we have $\norm{\lambda_0 a_0 + \lambda_1 a_1 - a_2}\leq 1/n$ for all $n<\omega$, so $\A\models P(n\lambda_0 a_0 + n\lambda_1 a_1 - na_2)$ and therefore also $\B\models P(e^{-\varepsilon}(n\lambda_0 b_0 + n\lambda_1 b_1 - nb_2))$, i.e. $\norm{\lambda_0 b_0 + \lambda_1 b_1 - b_2}\leq e^\varepsilon/n$ for all $n<\omega$. Thus $b_2 = \lambda_0 b_0 + \lambda_1 b_1$. As $f$ is continuous and $D$-linear, it is also $\R$-linear.
    \end{proof}

    \begin{lemma}
        \label{Lemma: Banach-Mazur perturbation distance}
        For any $\A\in\BS$, $k<\omega$ and $\varepsilon\geq 0$, there exists $\delta>0$ such that for any $\varphi(\bar{v})\in\good(k)$ and $\bar{a},\bar{b}\in\A^k$, if $d(\bar{a},\bar{b})<\delta$, then
        \[
            \A\models\varphi(\bar{a}) \implies \A\models\appr(\varphi,\varepsilon)(\bar{b}).
        \]
    \end{lemma}
    \begin{proof}
        Let $k<\omega$ and $\varepsilon>0$ be arbitrary and let $\delta_P = (e^\varepsilon-1)/k^2$ and $\delta_Q = (1-e^{-\varepsilon})/k^2$. Suppose that $\abs{c_i}\leq k$, and let $\bar{a},\bar{b}\in\A^k$. Now it is easy to see that whenever $d(\bar{a},\bar{b})<\delta_P$, i.e. $\norm{a_i-b_i}<\delta_P$ for $i<k$, we have
        \[
            \A\models P\left(\sum_{i=0}^{k-1}c_i a_i\right) \implies \A\models P\left(e^{-\varepsilon}\sum_{i=0}^{k-1}c_ib_i\right),
        \]
        and whenever $d(\bar{a},\bar{b})<\delta_Q$, i.e. $\norm{a_i-b_i}<\delta_Q$ for $i<k$, we have
        \[
            \A\not\models Q\left(e^{\varepsilon}\sum_{i=0}^{k-1}c_ib_i\right) \implies \A\not\models Q\left(\sum_{i=0}^{k-1}c_i a_i\right).
        \]
        Thus $\delta=\min\{\delta_P,\delta_Q\}$ suffices.
    \end{proof}

    \begin{lemma}
        \label{Lemma: Banach-Mazur tail-dense sequences behave nicely}
        For sequences $(a_n)_{n<\omega}\in\A^\omega$ and $(b_n)_{n<\omega}\in\B^\omega$ and a sequence $(\varepsilon_n)_{n<\omega}$ of positive numbers converging to $\varepsilon>0$, if for all $k<\omega$ and $\varphi(v_0,\dots,v_{k-1})\in\good(k)$ we have
        \[
            \A\models\varphi(a_{i_0},\dots,a_{i_{k-1}}) \implies \B\models\appr(\varphi,{\varepsilon_k})(b_{i_0},\dots,b_{i_{k-1}})
        \]
        for all  $i_0,\dots,i_{k-1}\geq k$, then for any increasing sequence $(i_n)_{n<\omega}$ of indices, $(a_{i_n})_{n<\omega}$ is Cauchy if and only if $(b_{i_n})_{n<\omega}$ is Cauchy.
    \end{lemma}
    \begin{proof}
        Let $(i_n)_{n<\omega}$ be an increasing sequence of indices and suppose $(a_{i_n})_{n<\omega}$ is Cauchy. Let $r>0$, and let $k\geq 2$ be large enough that $e^{\varepsilon+1}/r\leq k$, $\varepsilon_k-\varepsilon\leq 1$ and for all $m$ and $n$ such that $i_m,i_n\geq k$, we have and $\norm{a_{i_m}-a_{i_n}}\leq re^{-(\varepsilon+1)}$. Then $P(\frac{e^{\varepsilon+1}}{r}v_0-\frac{e^{\varepsilon+1}}{r}v_1)\in\good(k)$, so $\A\models P(\frac{e^{\varepsilon+1}}{r}a_{i_m}-\frac{e^{\varepsilon+1}}{r}a_{i_n})$ implies $\B\models P(e^{-\varepsilon_k}(\frac{e^{\varepsilon+1}}{r}b_{i_m}-\frac{e^{\varepsilon+1}}{r}b_{i_n}))$, whence $\norm{b_{i_m}-b_{i_n}}\leq e^{\varepsilon_k - \varepsilon - 1}r \leq e^{0}r = r$ for large enough $m$ and $n$, proving that $(b_{i_n})_{n<\omega}$ is Cauchy. In the other direction one can use exactly the same argument by Lemma~\ref{Lemma: Banach-Mazur vice versa follows if satisfaction is approximate}.
    \end{proof}

    \begin{definition}
        The \emph{Banach--Mazur distance} $\dbm(X,Y)$ of two Banach spaces $X$ and $Y$ is the infimum of all $\varepsilon\geq 0$ such that there exists a linear $e^\varepsilon$-bi-Lipschitz homeomorphism $f\colon X\to Y$.
    \end{definition}

    \begin{lemma}
        For $\A,\B\in\BS$, we have $\dbm(\A,\B)=d_{\isom}(\A,\B)$.
    \end{lemma}
    \begin{proof}
        Check the definitions.
    \end{proof}

    \subsection{Banach spaces and a correspondence notion}

    Define an approximation $\appr$ of $\LBb$ by setting $\appr(P_r(t),\varepsilon) = P_{r\dotplus\varepsilon}(t)$ and $\appr(Q_r(t),\varepsilon) = Q_{r\dotminus\varepsilon}(t)$,\footnote{By $x\dotplus y$ we denote $\min\{x+y, 1\}$.} and $\weakneg$ by setting $\weakneg(P_r(t)) = Q_r(t)$ and $\weakneg(Q_r(t)) = P_r(t)$. Then $\LBb$ with $\appr$ and $\weakneg$ is a vocabulary with approximations with respect to $\BSb$.

    We let $\corr_\varepsilon$ be the class of all set correspondences $R$ between a dense set of $\A$ and a dense set of $\B$, for $\A,\B\in\BSb$, such that there exist linear isometric embeddings $\iota_\A\colon\A\to\model C$ and $\iota_\B\colon\B\to\model C$ into a third space $\model{C}\in\BSb$ and for all $(a,b)\in R$
    \[
        \norm{\iota_\A(a) - \iota_\B(b)} \leq \varepsilon.
    \]
    Then it is not difficult to check that $\corr=(\corr_\varepsilon)_{\varepsilon\geq 0}$ is a correspondence notion of $\BSb$, except for item~\ref{item: Preserving atomic formulae suffices for epsilon-correspondence} of Definition~\ref{Definition: Correspondence notion} which is annoyingly tricky to verify and proven in Corollary~\ref{Corollary: Kadets formulae vs correspondences}.

    For convenience, denote by $\Lambda$ the set $\bigcup_{n<\omega}\{ \bar{\lambda}\in D^n \mid \sum_{i<n}\abs{\lambda_i}=1 \}$.
    \begin{lemma}
        \label{Lemma: Kadets epsilon-correspondence epsilon-preserves formulae}
        Let $\A,\B\in\BSb$, and let $R\subseteq\A\times\B$. If $R\in\corr_\varepsilon$, then, whenever $(a_i,b_i)\in R$, $i<n$, and $\varphi(v_0,\dots,v_{n-1})$ is an atomic formula, we have
        \[
            \A\models\varphi(\bar{a}) \implies \B\models\appr(\varphi,\varepsilon)(\bar{b}).
        \]
    \end{lemma}
    \begin{proof}
        As $R$ is an $\varepsilon$-correspondence between $\A$ and $\B$, there are linear isometric embeddings $\iota_\A\colon\A\to\model C$ and $\iota_\B\colon\B\to\model C$ such that for all $(a,b)\in R$
        \[
            \norm{\iota_\A(a) - \iota_\B(b)} \leq \varepsilon.
        \]
        Note that every atomic formula with variables among $v_0,\dots,v_{n-1}$ is equivalent to a formula of the form $P_r(\sum_{i<n}\lambda_i v_i)$ or $Q_r(\sum_{i<n}\lambda_i v_i)$ for some $r>0$ and $(\lambda_0,\dots,\lambda_{n-1})\in\Lambda$. Let $(a_i,b_i)\in R$, $i<n$, and $(\lambda_0,\dots,\lambda_{n-1})\in\Lambda$, and suppose that $\A\models P_r\left(\sum_{i<n}\lambda_i a_i\right)$, i.e. $\norm{\sum_{i<n}\lambda_i a_i}\leq r$. Now
        \begin{align*}
            \norm{\sum_{i<n}\lambda_i b_i} &= \norm{\sum_{i<n}\lambda_i \iota_\A(a_i) + \sum_{i<n}\lambda_i \iota_\B(b_i) - \sum_{i<n}\lambda_i \iota_\A(a_i)} \\
            &\leq \norm{\sum_{i<n}\lambda_i a_i} + \sum_{i<n}\abs{\lambda_i}\norm{\iota_\B(b_i) - \iota_\A(a_i)} \\
            &\leq r + \varepsilon,
        \end{align*}
        whence $\B\models P_{r+\varepsilon}(\sum_{i<n}\lambda_i b_i)$, i.e. $\B\models \appr(P_r(\sum_{i<n}\lambda_i v_i),\varepsilon)(\bar{b})$.

        The case for $Q_r(\sum_{i<n}\lambda_i v_i)$ is similar.
    \end{proof}

    \begin{definition}
        The \emph{Kadets distance} $\dk(X,Y)$ of two Banach spaces $X$ and $Y$ is the infimum of
        \[
            \dH{Z}(\iota_X(B_X),\iota_Y(B_Y))
        \]
        over all linear isometric embeddings $\iota_X\colon X\to Z$ and $\iota_Y\colon Y\to Z$ into a third Banach space $Z$, where $B_X$ and $B_Y$ are the closed unit balls of the respective spaces.
    \end{definition}

    For $\A,\B\in\BSb$, we write simply $\dk(\A,\B)$ for the Kadets distance between the Banach spaces spanned by $\A$ and $\B$.

    \begin{lemma}
        \label{Lemma: Kadets distance is captured by correspondence notion}
        For $\A,\B\in\BSb$, the following are equivalent for all $\varepsilon>0$.
        \begin{enumerate}
            \item $\dk(\A,\B)<\varepsilon$.
            \item There is $\varepsilon'\in(0,\varepsilon)$ such that, for any dense sets $\dense\A\subseteq\A$ and $\dense\B\subseteq\B$, there is $R\in\corr_{\varepsilon'}$ such that $\dom(R)=\dense\A$ and $\ran(R)=\dense\B$.
        \end{enumerate}
    \end{lemma}
    \begin{proof}
        First suppose that $\dk(\A,\B)<\varepsilon$. Pick any $\varepsilon'>0$ such that $\dk(\A,\B)<\varepsilon'<\varepsilon$. Then there are linear isometric embeddings $\iota_\A\colon\A\to\model C$ and $\iota_\B\colon\B\to\model C$ such that
        \[
            \dH{\model C}(\iota_\A(\A),\iota_\B(\B)) < \varepsilon'.
        \]
        Let $\dense\A\subseteq\A$ and $\dense\B\subseteq\B$ be any dense sets, and let
        \[
            R \coloneqq \{(a,b)\in\dense\A\times\dense\B : \norm{\iota_\A(a) - \iota_\B(b)} \leq \varepsilon' \}.
        \]
        It is easy to check that $\dom(R)=\dense\A$ and $\ran(R)=\dense\B$, and hence $R$ is as desired.

        Then suppose that $R\in\corr_{\varepsilon'}$ for some $\varepsilon'<\varepsilon$ and $\dom(R)=\dense\A$ and $\ran(R)=\dense\B$. Then there are linear isometric embeddings $\iota_\A\colon\A\to\model C$ and $\iota_\B\colon\B\to\model C$ such that
        \[
            \norm{\iota_\A(a)-\iota_\B(b)} \leq \varepsilon'
        \]
        for all $(a,b)\in R$. Denote $\delta = \varepsilon-\varepsilon'$.
        Using denseness of $\dom(R)$ and $\ran(R)$ respectively, we can show that $\sup_{a\in\A}d(\iota_\A(a),\iota_\B(\B))\leq\varepsilon' + \delta/2$ and $\sup_{b\in\B}d(\iota_\A(\A),\iota_\B(b))\leq\varepsilon' + \delta/2$, whence $\dH{\model C}(\iota_\A(\A),\iota_\B(\B))\leq\varepsilon'+\delta/2$. Then $\dk(\A,\B)\leq\varepsilon'+\delta/2 = \varepsilon' + (\varepsilon-\varepsilon')/2 < \varepsilon$.
    \end{proof}

    \begin{corollary}
        For $\A,\B\in\BSb$, we have $\dk(\A,\B)=d_\corr(\A,\B)$.
    \end{corollary}

    \begin{lemma}
        \label{Lemma: Kadets formulae maximum norm thingamajig}
        Let $\A,\B\in\BSb$. Let $R\subseteq\A\times\B$ be a set correspondence between a dense set of $\A$ and a dense set of $\B$, and suppose that there is $\delta>0$ such that whenever $(a_i,b_i)\in R$, $i<n$, we have
        \[
            \A\models\varphi(\bar{a}) \implies \B\models\appr(\varphi,\varepsilon-\delta)(\bar{b})
        \]
        for all atomic formulae $\varphi(\bar{v})$. Then there are $\delta',\delta''>0$ such that for all $(a_i,b_i)\in R$ with $\max\{\norm{a_i},\norm{b_i}\}>1-\delta''$, $i<n$, and $(\lambda_0,\dots,\lambda_{n-1})\in\Lambda$ we have
        \[
            \abs{\norm{\sum_{i<n}\lambda_i a_i} - \norm{\sum_{i<n}\lambda_i b_i}} \leq (\varepsilon-\delta')\sum_{i<n}\abs{\lambda_i}\max\{\norm{a_i},\norm{b_i}\}.
        \]
    \end{lemma}
    \begin{proof}
        Let $\delta'=\delta/2$ and $\delta'' = \delta/(2\varepsilon - \delta)$.
        Let $(a_i,b_i)\in R$, $i<n$, be such that $\max\{\norm{a_i},\norm{b_i}\}>1-\delta''$, and let $(\lambda_0,\dots,\lambda_{n-1})\in\Lambda$. Denote $r=\norm{\sum_{i<n}\lambda_i a_i}$. Then
        \[
            \A\models P_{r}\left( \sum_{i<n}\lambda_i a_i \right) \quad\text{and}\quad \A\models Q_{r}\left( \sum_{i<n}\lambda_i a_i \right),
        \]
        whence
        \[
            \B\models P_{r+\varepsilon-\delta}\left( \sum_{i<n}\lambda_i b_i \right) \quad\text{and}\quad \B\models Q_{r-\varepsilon+\delta}\left( \sum_{i<n}\lambda_i b_i \right).
        \]
        Then
        \begin{align*}
            \norm{\sum_{i<n}\lambda_i b_i} &\leq r + \varepsilon - \delta = \norm{\sum_{i<n}\lambda_i a_i} + (\varepsilon-\delta')(1-\delta'') \\
            &= \norm{\sum_{i<n}\lambda_i a_i} + (\varepsilon-\delta')\sum_{i<n}\abs{\lambda_i}(1-\delta'') \\
            &\leq \norm{\sum_{i<n}\lambda_i a_i} + (\varepsilon-\delta')\sum_{i<n}\abs{\lambda_i}\max\{\norm{a_i},\norm{b_i}\}
        \end{align*}
        and similarly $\norm{\sum_{i<n}\lambda_i b_i}\geq \norm{\sum_{i<n}\lambda_i a_i} - (\varepsilon-\delta')\sum_{i<n}\abs{\lambda_i}\max\{\norm{a_i},\norm{b_i}\}$.
        By combining these, we get
        \[
            \abs{\norm{\sum_{i<n}\lambda_i a_i} - \norm{\sum_{i<n}\lambda_i b_i}} \leq (\varepsilon-\delta')\sum_{i<n}\abs{\lambda_i}\max\{\norm{a_i},\norm{b_i}\}. \qedhere
        \]
    \end{proof}

    \begin{lemma}
        \label{Lemma: Kadets (epsilon-delta)-preserving formulae suffices for <epsilon-correspondence}
        Let $\A,\B\in\BSb$. Let $R\subseteq\A\times\B$ be a set correspondence between a dense set of $\A$ and a dense set of $\B$, and suppose that there is $\delta>0$ such that whenever $(a_i,b_i)\in R$, $i<n$, we have
        \[
            \A\models\varphi(\bar{a}) \implies \B\models\appr(\varphi,\varepsilon-\delta)(\bar{b})
        \]
        for all atomic formulae $\varphi(\bar{v})$. Then there is $\varepsilon'\in(0,\varepsilon)$ and $R\in\corr_{\varepsilon'}$ such that $\dom(R')=\dom(R)$ and $\ran(R')=\ran(R)$.
    \end{lemma}
    \begin{proof}
        The main idea of this proof is the same as in the proof of Lemma 18 in~\cite{CDK20}.
    
        Let $\delta',\delta''>0$ be as in Lemma~\ref{Lemma: Kadets formulae maximum norm thingamajig} and denote by $\hat{R}$ the set of all $(a,b)\in R$ such that $\max\{\norm{a},\norm{b}\}>1-\delta''$. 

        Let $A$ be the linear span of $\dom(\hat{R})$ and $B$ the linear span of $\ran(\hat{R})$. Define a function $\norm{\cdot}_C$ on the vector space $C \coloneqq A \oplus B$ by letting $\norm{(a,b)}_{C}$ be the infimum of
        \[
             \rho(\norm{a_n} + \norm{b_n} + (\varepsilon-\delta')\sum_{i<n}\abs{\lambda_i}\max\{\norm{a_i},\norm{b_i}\})
        \]
        over all $n<\omega$, $a_0,\dots,a_{n-1}\in\dom(\hat{R})$, $a_n\in A$, $b_0,\dots,b_{n-1}\in\ran(\hat{R})$, $b_n\in B$, $\rho>0$ and $(\lambda_0,\dots,\lambda_{n-1})\in\Lambda$ such that
        \begin{itemize}
            \item $\max\{\norm{a},\norm{b}\}\leq\rho$,
            \item $(a_i,b_i)\in\hat{R}$ for $i<n$,
            \item $a = \rho(a_n + \sum_{i<n}\lambda_i a_i)$ and
            \item $b = \rho(b_n - \sum_{i< n}\lambda_i b_i)$.
        \end{itemize}
        It is straightforward (although tedious) to verify that $\norm{\cdot}_{C}$ is a norm.

        Next we claim that $\norm{(a,0)}_C = \norm{a}$ for all $a\in A$ and $\norm{(0,b)}_C = \norm{b}$ for all $b\in B$. We show the first claim; the second claim is proved with an analogous argument. Fix $a\in A$. It is clear that $\norm{(a,0)}_C\leq\norm{a}$, as we can write $a = \rho(a_n + \sum_{i<n}\lambda_i a_i)$ and $0 = \rho(b_n - \sum_{i<n}\lambda_i b_i)$ for $n=0$ and $\rho=\norm{a}$ by choosing $a_0 = a/\rho$ and $b_0 = 0$, whence
        \[
            \norm{(a,0)}_C \leq \rho(\norm{a_0} + \norm{b_0}) = \norm{\rho a_0} = \norm{a}.
        \]
        To show that $\norm{a}\leq\norm{(a,0)}_C$, pick arbitrary $\rho\geq\max\{\norm{a}\}$, $\bar\lambda\in\Lambda$ and $(a_i,b_i)\in\hat{R}$, $i<n$, $a_n\in A$ and $b_n\in B$, such that $a = \rho(a_n + \sum_{i<n}\lambda_i a_i)$ and $0 = \rho(b_n - \sum_{i<n}\lambda_i b_i)$.
        Now by Lemma~\ref{Lemma: Kadets formulae maximum norm thingamajig}, as $(a_i,b_i)\in R$ and $\max\{\norm{a_i},\norm{b_i}\}>1-\delta''$, we have
        \[
            \abs{\norm{\sum_{i<n}\lambda_i a_i} - \norm{\sum_{i<n}\lambda_i b_i}} \leq (\varepsilon-\delta')\sum_{i<n}\abs{\lambda_i}\max\{\norm{a_i},\norm{b_i}\}.
        \]
        Then
        \begin{align*}
            \rho(\norm{a_n} &+ \norm{b_n} + (\varepsilon-\delta')\sum_{i<n}\abs{\lambda_i}\max\{\norm{a_i},\norm{b_i}\}) \\
            &= \rho\left( \norm{a_n} + \norm{\sum_{i<n}\lambda_i b_i} + (\varepsilon-\delta')\sum_{i<n}\abs{\lambda_i}\max\{\norm{a_i},\norm{b_i}\} \right) \\
            &\geq \rho\left( \norm{a_n} + \norm{\sum_{i<n}\lambda_i a_i} \right) \\
            &\geq \norm{a},
        \end{align*}
        whence, as $a_i$, $b_i$ and $\lambda_i$ were arbitrary, we get $\norm{a}\leq\norm{(a,0)}_C$. This shows that the embeddings $a\mapsto (a,0)$ and $b\mapsto (0,b)$ are isometric. Let $\model C$ be the completion of $C$. Then the Banach spaces spanned by $\A$ and $\B$ isometrically embed into subspaces of $\model C$.

        Let $a\in\A$. Note that the set
        \[
            \left\{\sum_{i<n}\lambda_i a_i \ \right|\left.\vphantom{\sum_{i<n}}\, a_0,\dots,a_{n-1}\in\dom(\hat{R}),\bar\lambda\in\Lambda \right\}
        \]
        is dense in $\A$, so we can find $\bar\lambda\in\Lambda$ and $a_i\in\dom(\hat{R})$, $i<n$, such that $\norm{a - \sum_{i<n}\lambda_i a_i} < \delta'/2$. Let $b_i$ be such that $(a_i,b_i)\in\hat{R}$. Let $a'=\sum_{i<n}\lambda_i a_i$ and $b=\sum_{i<n}\lambda_i b_i$. By denoting $a_n = 0$, $b_n = 0$ and $\rho=1$, we can write $a' = \rho(a_n + \sum_{i<n}\lambda_i a_i)$ and $-b = \rho(b_n - \sum_{i<n}\lambda'_i b_i)$. Then
        \begin{align*}
            \norm{(a',-b)}_C &\leq \rho(\norm{a_n} + \norm{b_n} + (\varepsilon-\delta')\sum_{i<n}\abs{\lambda_i}\max\{\norm{a_i},\norm{b_i}\}) \\
            &= 0 + 0 + (\varepsilon-\delta')\sum_{i<n}\abs{\lambda_i}\max\{\norm{a_i},\norm{b_i}\} \\
            &\leq \varepsilon-\delta',
        \end{align*}
        whence
        \begin{align*}
            \norm{(a,-b)}_C &\leq \norm{a - a'} + \norm{(a',-b)}_C \\
            &\leq \delta'/2 + \varepsilon - \delta' \\
            &= \varepsilon - \delta'/2.
        \end{align*}
        Similarly for any $b\in\B$ we can find $a\in\A$ with $\norm{(a,-b)}_C\leq\varepsilon-\delta'/2$. Thus we have established that $\dk(\hat\A,\hat\B)<\varepsilon$. Then the existence of our desired relation in $\corr_{\varepsilon'}$ directly follows from Lemma~\ref{Lemma: Kadets distance is captured by correspondence notion}.
    \end{proof}

    \begin{corollary}
        \label{Corollary: Kadets formulae vs correspondences}
        Let $\A,\B\in\BSb$ and let $\dense\A\subseteq\A$ and $\dense\B\subseteq\B$ be dense sets of the respective spaces. Then for any $\varepsilon>0$ the following are equivalent.
        \begin{enumerate}
            \item There is $\varepsilon'\in(0,\varepsilon)$ and $R\in\corr_{\varepsilon'}$ such that $\dom(R)=\dense\A$ and $\ran(R)=\dense\B$.
            \item There is $\varepsilon'\in(0,\varepsilon)$ and a set correspondence $R$ between $\dense\A$ and $\dense\B$ such that whenever $(a_i,b_i)\in R$, $i<n$, and $\varphi(v_0,\dots,v_{n-1})$ is an atomic formula, we have
            \[
                \A\models\varphi(\bar{a}) \implies \B\models\appr(\varphi,\varepsilon')(\bar{b}).
            \]
        \end{enumerate}
    \end{corollary}
    \begin{proof}
        Put together Lemmata~\ref{Lemma: Kadets epsilon-correspondence epsilon-preserves formulae} and~\ref{Lemma: Kadets (epsilon-delta)-preserving formulae suffices for <epsilon-correspondence}.
    \end{proof}

    \section{Games}\label{Section: Games}

    For the rest of this paper, we fix a vocabulary $L$ and a class $\class$ of $L$-structures such that $L$ has approximations for $\class$.

    \subsection{Function Games}

     Let $\isom=(\isom_{\varepsilon})_{\varepsilon\geq 0}$ be an isomorphism notion of $\class$ such that $L$ is good for $(\class,\isom)$. We define two variants of the Ehrenfeucht--Fraïssé game that are used to capture the concept of being $\varepsilon$-isomorphic.

    \begin{definition}
        Let $\A,\B\in\class$, and let $\dense{\A}$ and $\dense{\B}$ be dense subsets of the respective spaces. For $\varepsilon\geq 0$, $\bar{\varepsilon}=(\varepsilon_0,\dots,\varepsilon_{n-1})\in \open{\varepsilon,\infty}^n$ and tuples $(a_0,\dots,a_{n-1})\in\A^n$ and $(b_0,\dots,b_{n-1})\in\B^n$, we denote by
        \[
            \EF{\omega,\varepsilon,\bar{\varepsilon}}{\A,\B}{(\dense{\A},(a_0,\dots,a_{n-1})),(\dense{\B},(b_0,\dots,b_{n-1}))}
        \]
        the two-player game of length $\omega$, defined as follows. Denote the players by $\playerone$ and $\playertwo$. At each round $k<\omega$, first $\playerone$ picks an element $x_{k+n}\in\dense{\A}\cup\dense{\B}$ and in addition a positive real number $\varepsilon_{k+n}>\varepsilon$, and then $\playertwo$ responds with some $y_{k+n}\in\dense{\A}\cup\dense{\B}$ so that if $x_{k+n}$ was in $\dense{\A}$, then $y_{k+n}$ is in $\dense{\B}$, and vice versa. We call this game the Ehrenfeucht--Fraïssé game of length $\omega$ and precision $\varepsilon$, between the sets $\dense{\A}$ and $\dense{\B}$, with starting position $((a_0,\dots,a_{n-1}),(b_0,\dots,b_{n-1}),\bar{\varepsilon})$.

        For $k\geq n$, we denote by $a_k$ the element of the set $\{x_k,y_k\}$ that belongs to $\dense{\A}$ and by $b_k$ the one that belongs to $\dense{\B}$, hence expanding the original tuples $(a_0,\dots,a_{n-1})$ and $(b_0,\dots,b_{n-1})$ into infinite sequences $(a_i)_{i<\omega}$ and $(b_i)_{i<\omega}$.

        Player $\playertwo$ wins a play of the game if for all $k<\omega$ and $\varphi(v_0,\dots,v_{m-1})\in\good(k)$,
        \[
            \A\models\varphi(a_{i_0},\dots,a_{i_{m-1}}) \implies \B\models\appr(\varphi,{\varepsilon_k})(b_{i_0},\dots,b_{i_{m-1}})
        \]
        for all $i_0,\dots,i_{m-1}\geq k$.

        If $n=0$, then we denote the game by $\EF{\omega,\varepsilon}{\A,\B}{\dense{\A},\dense{\B}}$.
    \end{definition}

    \begin{remark}
        The game $\EF{\omega,\varepsilon,\bar{\varepsilon}}{\A,\B}{(\dense{\A},(a_0,\dots,a_{n-1})),(\dense{\B},(b_0,\dots,b_{n-1}))}$ corresponds to a position in the game $\EF{\omega,\varepsilon}{\A,\B}{\dense{\A},\dense{\B}}$ on round $n$, where $a_i$, $b_i$ and $\varepsilon_i$, $i<n$, have been played by the players.
    \end{remark}

    \begin{definition}
        Let $\A,\B\in\class$, and let $\dense{\A}$ and $\dense{\B}$ be dense subsets of the respective spaces. For an ordinal $\alpha$ and a number $\varepsilon\geq 0$, and $\bar{\varepsilon}=(\varepsilon_0,\dots,\varepsilon_{n-1})\in \open{\varepsilon,\infty}^n$ and $\bar{k}=(k_0,\dots,k_{n-1})\in\omega^n$, and tuples $(a_0,\dots,a_{n-1})\in\A^n$ and $(b_0,\dots,b_{n-1})\in\B^n$, we define the game
        \[
            \EFD{\alpha,\varepsilon,\bar{\varepsilon},\bar{k}}{\A,\B}{(\dense{\A},a_0,\dots,a_{n-1}),(\dense{\B},b_0,\dots,b_{n-1})}
        \]
        exactly as $\EF{\omega,\varepsilon,\bar{\varepsilon}}{\A,\B}{(\dense{\A},(a_0,\dots,a_{n-1})),(\dense{\B},(b_0,\dots,b_{n-1}))}$, but on each round $i$, in addition to choosing $x_{i+n}\in\dense\A\cup\dense\B$ and $\varepsilon_{i+n}>\varepsilon$, $\playerone$ also chooses some $k_{i+n}<\omega$ and some ordinal $\alpha_i<\alpha$ so that for all $i<j$, $\alpha_j<\alpha_i$. A play ends after the round $i$ when $\playerone$ chooses the ordinal $\alpha_i = 0$. We call this game the dynamic Ehrenfeucht--Fraïssé game of precision $\varepsilon$ and clock $\alpha$, between the sets $\dense{\A}$ and $\dense{\B}$, with starting position $((a_0,\dots,a_{n-1}),(b_0,\dots,b_{n-1}),\bar{\varepsilon},\bar{k})$.

        Player $\playertwo$ wins a play if for each $i$, for all $\varphi(v_0,\dots,v_{m-1})\in\good(k_i)$ we have
        \[
            \A\models\varphi(a_{j_0},\dots,a_{j_{m-1}}) \implies \B\models\appr(\varphi,{\varepsilon_i})(b_{j_0},\dots,b_{j_{m-1}})
        \]
        for all $j_0,\dots,j_{m-1}\geq i$.

        If $n=0$, we denote the game by $\EFD{\alpha,\varepsilon}{\A,\B}{\dense{\A},\dense{\B}}$.
    \end{definition}

    \begin{remark}
        The game $\EFD{\alpha,\varepsilon,\bar{\varepsilon},\bar{k}}{\A,\B}{(\dense{\A},a_0,\dots,a_{n-1}),(\dense{\B},b_0,\dots,b_{n-1})}$ corresponds to a position in the game $\EFD{\alpha+n,\varepsilon}{\A,\B}{\dense{\A},\dense{\B}}$ on round $n$, where $a_i$, $b_i$, $k_i$ and $\varepsilon_i$, $i<n$, have been played by the players, and on round $0\leq i<n$ $\playerone$ has played $\alpha_i=\alpha+(n-1-i)$.
    \end{remark}

    Winning strategies for each player are defined as usual. If $\playerone$ has a winning strategy in a game $G$, we write $\playerone\wins G$, and if $\playertwo$ has a winning strategy, we write $\playertwo\wins G$. We immediately make the following observations.
    \begin{lemma}
        \begin{enumerate}
            \item Both $\EF{\omega,\varepsilon,\bar{\varepsilon}}{\A,\B}{(\dense{\A},(a_0,\dots,a_{n-1})),(\dense{\B},(b_0,\dots,b_{n-1}))}$ and $\EFD{\alpha,\varepsilon,\bar{\varepsilon},\bar{k}}{\A,\B}{(\dense{\A},a_0,\dots,a_{n-1}),(\dense{\B},b_0,\dots,b_{n-1})}$ are determined.

            \item If
            \[
                \playertwo\wins\EFD{\alpha,\varepsilon,\bar{\varepsilon},\bar{K}}{\A,\B}{(\dense{\A},a_0,\dots,a_{n-1}),(\dense{\B},b_0,\dots,b_{n-1})}
            \]
            and $\beta\leq \alpha$, $\delta\geq\varepsilon$, $\delta_i\geq\varepsilon_i$ and $k_i\leq K_i$, then also
            \[
                \playertwo\wins\EFD{\beta,\delta,\bar{\delta},\bar{k}}{\A,\B}{(\dense{\A},a_0,\dots,a_{n-1}),(\dense{\B},b_0,\dots,b_{n-1})}.
            \]

            \item If $\alpha$ is a limit and
            \[
                \playertwo\wins\EFD{\beta,\varepsilon,\bar{\varepsilon},\bar{k}}{\A,\B}{(\dense{\A},a_0,\dots,a_{n-1}),(\dense{\B},b_0,\dots,b_{n-1})}
            \]
            for all $\beta<\alpha$, then
            \[
                \playertwo\wins\EFD{\alpha,\varepsilon,\bar{\varepsilon},\bar{k}}{\A,\B}{(\dense{\A},a_0,\dots,a_{n-1}),(\dense{\B},b_0,\dots,b_{n-1})}.
            \]
        \end{enumerate}
    \end{lemma}

    We then proceed to show that each player, when having a winning strategy in a dynamic game between two spaces, can alter their strategy so that they only play in dense subsets.

    \begin{lemma}
        \label{Lemma: II can play in a dense set}
        Suppose that $\playertwo\wins\EFD{\alpha,\varepsilon,\bar{\varepsilon},\bar{k}}{\A,\B}{(\A,a_0,\dots,a_{n-1}),(\B,b_0,\dots,b_{n-1})}$ and $\dense{\A}\subseteq\A$ and $\dense{\B}\subseteq\B$ are dense. Then there is a winning strategy for $\playertwo$ in the game $\EFD{\alpha,\varepsilon,\bar{\varepsilon},\bar{k}}{\A,\B}{(\A,a_0,\dots,a_{n-1}),(\B,b_0,\dots,b_{n-1})}$ whose range is contained in $\dense{\A}\cup\dense{\B}$.
    \end{lemma}
    \begin{proof}
        Let $\tau$ be a winning strategy for $\playertwo$ and fix for any $\varepsilon'>0$ and $k<\omega$ a number $\delta(\varepsilon',k)$ that is a perturbation distance for $k$ and $\varepsilon'$ in both spaces and in addition increasing with respect to $\varepsilon'$ and decreasing with respect to $k$ (such a function exists by Lemma~\ref{Lemma: Perturbation distance function}). We define a new winning strategy $\tau'$ whose range is contained in the union of the dense sets. If $x_i$, $\varepsilon_i$, $\alpha_i$ and $k_i$ are valid moves of $\playerone$ for $i\leq m$, we let $\tau'((x_0,\varepsilon_0,k_0,\alpha_0),\dots,(x_{m},\varepsilon_{m},k_{m},\alpha_{m}))$ be any such $y_m$ that
        \begin{enumerate}
            \item if $x_m\in\A$, then $y_m\in\dense{\B}$, and if $x_m\in\B$, then $y_m\in\dense{\A}$,
            \item $d(y_m,z)<\delta((\varepsilon_i-\varepsilon)/3, \tilde{k}_i)$ for all $i\leq m$, where
            \[
                z=\tau((x_0,\varepsilon+(\varepsilon_0-\varepsilon)/3,\tilde{k}_0,\alpha_0),\dots,(x_m,\varepsilon + (\varepsilon_m-\varepsilon)/3,\tilde{k}_m,\alpha_m))
            \]
            and $\tilde{k}_i=k_i\ceil{e^{\varepsilon_i}}$.
        \end{enumerate}

        We now prove that this works. Let $((x_i,\varepsilon_i,k_i,\alpha_i),y_i)_{i<l}$ be an arbitrary play where $\playertwo$ has used $\tau'$. Suppose for a contradiction that $\playerone$ has won. Then there are $i<l$ and $\varphi(v_0,\dots,v_{p-1})\in\good(k_i)$ such that
        \[
            \A\models\varphi(a_{j_0},\dots,a_{j_{p-1}}) \text{ but } \B\not\models\appr(\varphi,{\varepsilon_i})(b_{j_0},\dots,b_{j_{p-1}})
        \]
        for some $j_0,\dots,j_{p-1}\geq i$. Now we look at an auxiliary play where $\playertwo$ uses $\tau$ and $\playerone$ plays the moves $(x_m,\varepsilon+(\varepsilon_m-\varepsilon)/3,\tilde{k}_m,\alpha_m)$. Let $z_m$, $m<l$, be the moves of $\playertwo$ according to $\tau$ in the auxiliary game, and denote by $\hat{a}_m$ and $\hat{b}_m$ the elements of $\A$ and $\B$, respectively, played in this play.

        By the definition of $\tau'$, we have $d(y_m,z_m)<\delta((\varepsilon_i-\varepsilon)/3,\tilde{k}_i)$ for all $m\geq i$, and thus $d(a_m,\hat{a}_m)<\delta((\varepsilon_i-\varepsilon)/3,\tilde{k}_i)$ for all applicable $m\geq i$. Since $\tilde{k}_i\geq k_i$, by Definition~\ref{Definition: Good vocabulary (isomorphism)}~\ref{item: k-good formulae increase when k increases}, $\varphi\in\good(\tilde{k}_i)$, and as
        \[
            \A\models\varphi(a_{j_0},\dots,a_{j_{p-1}}),
        \]
        by Definition~\ref{Definition: Good vocabulary (isomorphism)}~\ref{item: Perturbation distance}, we have
        \[
            \A\models\appr(\varphi,{(\varepsilon_i-\varepsilon)/3})(\hat{a}_{j_0},\dots,\hat{a}_{j_{p-1}}).
        \]
        As by Definition~\ref{Definition: Good vocabulary (isomorphism)}~\ref{item: An approximating of a k-good formula is still k-good for a little big bigger k} and~\ref{item: k-good formulae increase when k increases} $\appr(\varphi,{(\varepsilon_i-\varepsilon)/3})\in\good(\tilde{k}_i)$ and $\tau$ is a winning strategy, we have
        \[
            \B\models\appr(\appr(\varphi,{(\varepsilon_i-\varepsilon)/3}),{\varepsilon+(\varepsilon_i-\varepsilon)/3})(\hat{b}_{j_0},\dots,\hat{b}_{j_{p-1}}).
        \]
        As $\appr(\appr(\varphi,{(\varepsilon_i-\varepsilon)/3}),{\varepsilon+(\varepsilon_i-\varepsilon)/3}) = \appr(\varphi,{\varepsilon+2(\varepsilon_i-\varepsilon)/3})$, the formula is still in $\good(\tilde{k}_i)$. Since $d(y_m, z_m)<\delta((\varepsilon_i-\varepsilon)/3,\tilde{k}_i)$ for all $m\geq i$ and thus $d(b_m, \hat{b}_m)<\delta((\varepsilon_i-\varepsilon)/3,\tilde{k}_i)$, we have
        \[
            \B\models\appr(\appr(\varphi,{\varepsilon+2(\varepsilon_i-\varepsilon)/3}),{(\varepsilon_i-\varepsilon)/3})(b_{j_0},\dots,b_{j_{p-1}}).
        \]
        As $\appr(\appr(\varphi,{\varepsilon+2(\varepsilon_i-\varepsilon)/3}),{(\varepsilon_i-\varepsilon)/3}) = \appr(\varphi,{\varepsilon_i})$, this is a contradiction.
    \end{proof}

    \begin{corollary}
        \label{Corollary: II wins in the dense sets if she wins in the models}
        If $\playertwo$ wins $\EFD{\alpha,\varepsilon,\bar{\varepsilon},\bar{k}}{\A,\B}{(\A,a_0,\dots,a_{n-1}),(\B,b_0,\dots,b_{n-1})}$, then whenever $\dense{\A}\subseteq\A$ and $\dense{\B}\subseteq\B$ are dense subsets, $\playertwo$ also wins the game $\EFD{\alpha,\varepsilon,\bar{\varepsilon},\bar{k}}{\A,\B}{(\dense{\A},a_0,\dots,a_{n-1}),(\dense{\B},b_0,\dots,b_{n-1})}$.
    \end{corollary}

    \begin{lemma}
        \label{Lemma: I can play in a dense set}
        Suppose that $\playerone\wins\EFD{\alpha,\varepsilon,\bar{\varepsilon},\bar{k}}{\A,\B}{(\A,a_0,\dots,a_{n-1}),(\B,b_0,\dots,b_{n-1})}$ and $\dense{\A}\subseteq\A$ and $\dense{\B}\subseteq\B$ are dense. Then there is a winning strategy for $\playerone$ in the game $\EFD{\alpha,\varepsilon,\bar{\varepsilon},\bar{k}}{\A,\B}{(\A,a_0,\dots,a_{n-1}),(\B,b_0,\dots,b_{n-1})}$ whose range is contained in $\dense{\A}\cup\dense{\B}$.
    \end{lemma}
    \begin{proof}
        Similar to the proof of Lemma~\ref{Lemma: II can play in a dense set}.
    \end{proof}

    \begin{corollary}
        \label{Corollary: I wins in the dense sets if he wins in the models}
        If $\playerone$ wins $\EFD{\alpha,\varepsilon,\bar{\varepsilon},\bar{k}}{\A,\B}{(\A,a_0,\dots,a_{n-1}),(\B,b_0,\dots,b_{n-1})}$, then whenever $\dense{\A}\subseteq\A$ and $\dense{\B}\subseteq\B$ are dense subsets, $\playerone$ also wins the game $\EFD{\alpha,\varepsilon,\bar{\varepsilon},\bar{k}}{\A,\B}{(\dense{\A},a_0,\dots,a_{n-1}),(\dense{\B},b_0,\dots,b_{n-1})}$.
    \end{corollary}

    \begin{corollary}
        \label{Corollary: II wins in the dense sets iff she wins in the models}
        $\playertwo\wins\EFD{\alpha,\varepsilon,\bar{\varepsilon},\bar{k}}{\A,\B}{(\A,a_0,\dots,a_{n-1}),(\B,b_0,\dots,b_{n-1})}$ if and only if $\playertwo\wins\EFD{\alpha,\varepsilon,\bar{\varepsilon},\bar{k}}{\A,\B}{(\dense{\A},a_0,\dots,a_{n-1}),(\dense{\B},b_0,\dots,b_{n-1})}$, whenever $\dense{\A}\subseteq\A$ and $\dense{\B}\subseteq\B$ are dense.
    \end{corollary}
    \begin{proof}
        Putting together Corollaries~\ref{Corollary: II wins in the dense sets if she wins in the models} and~\ref{Corollary: I wins in the dense sets if he wins in the models} yields this result.
    \end{proof}

    The next theorem connects the infinite game between two spaces to the spaces being $\varepsilon$-isomorphic.

    \begin{theorem}
        \label{Theorem: If II wins the long game then there exists an epsilon-isomorphism}
        For separable $\A$ and $\B$, $\playertwo\wins\EF{\omega,\varepsilon}{\A,\B}{\A,\B}$ if and only if there exists an $\varepsilon$-isomorphism $\A\to\B$.
    \end{theorem}
    \begin{proof}
        If there exists an $\varepsilon$-isomorphism $f$, then by Definition~\ref{Definition: Good vocabulary (isomorphism)}~\ref{item: Preserving atomic formulae suffices for epsilon-isomorphism} $\playertwo$ wins the game by playing moves $f(a_n)$ when $\playerone$ chooses $a_n\in\A$ and $f^{-1}(b_n)$ when he chooses $b_n\in\B$.

        For the converse, let $\dense{\A}\subseteq\A$ and $\dense{\B}\subseteq\B$ be countable and dense. Let $(\varepsilon_i)_{i<\omega}$ be a sequence converging down to $\varepsilon$ and let $(x_i)_{i<\omega}$ be a sequence such that each element of $\dense{\A}\cup\dense{\B}$ occurs in the sequence infinitely many times. Let $\playerone$ play $(x_i,\varepsilon_i)$ on round $i$ in the game $\EF{\omega,\varepsilon}{\A,\B}{\A,\B}$, and let $y_i$ be the choices of $\playertwo$ following a winning strategy. As $\playertwo$ wins the play, we have for all $k<\omega$ and $\varphi\in\good(k)$
        \[
            \A\models\varphi(a_{i_0},\dots,a_{i_{m-1}}) \implies \B\models\appr(\varphi,{\varepsilon_k})(b_{i_0},\dots,b_{i_{m-1}})
        \]
        for any $i_0,\dots,i_{m-1}\geq k$.

        Now by Definition~\ref{Definition: Good vocabulary (isomorphism)}~\ref{item: Simultaneous Cauchy}, $(a_{i_n})_{n<\omega}$ is Cauchy if and only if $(b_{i_n})_{n<\omega}$ is Cauchy. From this it follows that whenever $(i_n)_{n<\omega}$ and $(j_n)_{n<\omega}$ are increasing sequences of indices, $\lim_{n\to\infty}a_{i_n}=\lim_{n\to\infty}a_{j_n}$ if and only if  $\lim_{n\to\infty}b_{i_n}=\lim_{n\to\infty}b_{j_n}$. Thus we can define a partial injection $\vartheta\colon\A\to\B$ by setting
        \[
            \vartheta(\lim_{n\to\infty}a_{i_n})=\lim_{n\to\infty}{b_{i_n}}.
        \]
        In fact, $\vartheta$ is a total bijection because it maps limit points of $(a_i)_{i<\omega}$ to limit points of $(b_i)_{i<\omega}$, and as these sequences visit every neighbourhood of every point infinitely often, their limit points cover the whole spaces.

        Last we note that $\vartheta$ is an $\varepsilon$-isomorphism. Let $k<\omega$, $\varphi(\bar{v})\in\good(k)$ and $\bar{c}\in\A^{\abs{\bar{v}}}$. Let $(i^j_n)_{n<\omega}$ be such that $a_{i^j_n}\to c_j$. Then by the definition of $\vartheta$,
        \[
            \vartheta(\bar{c})=(\lim_{n\to\infty}b_{i^0_n},\dots,\lim_{n\to\infty}b_{i^{\abs{\bar{v}}}_n}).
        \]
        By Lemma~\ref{Lemma: Limits satisfy k-good formulae} we then have
        \[
            \A\models\varphi(\bar{c}) \implies \B\models\appr(\varphi,\varepsilon)(\vartheta(\bar{c})).
        \]
        By Definition~\ref{Definition: Good vocabulary (isomorphism)}~\ref{item: Preserving atomic formulae suffices for epsilon-isomorphism}, this implies that $\vartheta\in\isom_\varepsilon$.
    \end{proof}

    \begin{theorem}
        \label{Theorem: If II wins the long game she wins all dynamic games}
        If $\playertwo\wins\EF{\omega,\varepsilon}{\A,\B}{\A,\B}$, then $\playertwo\wins\EFD{\alpha,\varepsilon}{\A,\B}{\A,\B}$ for all $\alpha\in\On$.
    \end{theorem}
    \begin{proof}
        The idea of $\playertwo$ in the game $\EFD{\alpha,\varepsilon}{\A,\B}{\A,\B}$ is to pretend that she is playing $\EF{\omega,\varepsilon}{\A,\B}{\A,\B}$ against a player that plays $\playerone$'s moves but in a rate slow enough so that the round number $k$ catches up with the $k_i$'s $\playerone$ plays in the $\alpha$-game. To be more precise, let $\tau$ be a winning strategy for $\playertwo$ in the long game. Define a strategy $\tau'$ for $\playertwo$ in the $\alpha$-game by setting $\tau'((x_0,\varepsilon_0,k_0,\alpha_0),\dots, (x_n,\varepsilon_n,k_n,\alpha_n))=\tau((\tilde{x}_i,\tilde{\varepsilon}_i)_{i\leq m})$,\footnote{Here we, for simplicity of notation, assume that $k_0<\dots<k_n$ and $\varepsilon_n\leq\dots\leq\varepsilon_0$.} where
        \begin{enumerate}
            \item $m=k_n$,
            \item $\tilde{x}_{k_i}=x_i$ and $\tilde{\varepsilon}_{k_i}=\varepsilon_i$ for $i\leq n$ and
            \item $\tilde{x}_j = \tilde{x}_{k_i}$ and $\tilde{\varepsilon}_j = \tilde{\varepsilon}_{k_i}$ for $k_i<j<k_{i+1}$.
        \end{enumerate}

        We now show that $\playertwo$ wins any play using this strategy. Suppose for a contradiction that $((x_i,\varepsilon_i,k_i,\alpha_i),y_i)_{i<l}$ is a play where $\playertwo$ has lost following the strategy. This means that there is $i<l$, $\varphi(\bar{v})\in\good(k_i)$ and some $j_0,\dots,j_{k_i-1}\geq i$ such that
        \[
            \A\models\varphi(a_{j_0},\dots,a_{j_{k_i-1}}) \quad\text{but}\quad \B\not\models\appr(\varphi,{\varepsilon_i})(b_{j_0},\dots,b_{j_{k_i-1}}).
        \]
        However, $y_n$ have been chosen according to $\tau$ in the game $\EF{\omega,\varepsilon}{\A,\B}{\A,\B}$, where $\playerone$ has played $(x_n,\varepsilon_n)$ at round $k_n$. Hence, as $\varphi\in\good(k_i)$, for any $n_0,\dots,n_{k_i-1}\geq k_i$
        \[
            \A\models\varphi(\tilde{a}_{n_0},\dots,\tilde{a}_{n_{k_i-1}}) \implies \B\models\appr(\varphi,{\tilde{\varepsilon}_{k_i}})(\tilde{b}_{n_0},\dots,\tilde{b}_{n_{k_i-1}}).
        \]
        Considering that $\tilde{\varepsilon}_{k_i}=\varepsilon_i$ and that each $a_{j_n}$ is some $\tilde{a}_m$ for $m\geq k_i$ such that $\A\models\varphi(a_{j_0},\dots,a_{j_{k_i-1}})$, we get that $\B\models\appr(\varphi,{\varepsilon_i})(b_{j_0},\dots,b_{j_{k_i-1}})$, a contradiction.
    \end{proof}

    \begin{theorem}
        \label{Theorem: If II wins many enough dynamic games then she wins the long game}
        If $\playertwo\wins\EFD{\alpha,\varepsilon}{\A,\B}{\A,\B}$ for all $\alpha<(\density(\A)+\density(\B))^+$, then $\playertwo\wins\EF{\omega,\varepsilon}{\A,\B}{\A,\B}$.
    \end{theorem}
    \begin{proof}
        Let $\dense{\A}\subseteq\A$ and $\dense{\B}\subseteq\B$ be dense sets of minimal cardinality and let $\kappa=\abs{\dense{\A}}+\abs{\dense{\B}}$. The strategy for $\playertwo$ in $\EF{\omega,\varepsilon}{\A,\B}{\A,\B}$ is to play such moves that the following condition holds at every round:
        \begin{itemize}
            \item[$(*)$] If the position in $\EF{\omega,\varepsilon}{\A,\B}{\A,\B}$ is $((x_0,\varepsilon_0),y_0,\dots,(x_{n-1},\varepsilon_{n-1}),y_{n-1})$, then for all $\alpha<\kappa^+$,
            \[
                \playertwo\wins\EFD{\alpha,\varepsilon,(\varepsilon_0,\dots,\varepsilon_{n-1}),(0,\dots,n-1)}{\A,\B}{(\A,a_0,\dots,a_{n-1}),(\B,b_0,\dots,b_{n-1})}.
            \]
        \end{itemize}

        Clearly such a strategy is winning if it exists. We show by induction on the round $n$ that $\playertwo$ can always play such moves that $(*)$ holds. If $n=0$, then no moves have been played yet and by initial assumption, $(*)$ is satisfied. Suppose that $(*)$ is true until round $n$ and, on round $n+1$, let $\playerone$ play the element $x_n$ and $\varepsilon_n$. Now $\playertwo$ will have to respond with some $y_n$ so that $(*)$ will remain true for the next position. By Lemma~\ref{Lemma: II can play in a dense set}, if $\playertwo$ has any chance of survival, she can pick her move from the dense set $\dense{\A}\cup\dense{\B}$. Let $Y$ be the set of all possible moves of $\playertwo$ in the dense set and assume that no $y_n\in Y$ will make $(*)$ true. This means that for every $y_n\in Y$ there is some ordinal $\alpha_{y_n}<\kappa^+$ such that
        \[
            \playertwo\doesnotwin\EFD{\alpha_{y_n},\varepsilon,(\varepsilon_0,\dots,\varepsilon_n),(0,\dots,n)}{\A,\B}{(\dense{\A},a_0,\dots,a_n),(\dense{\B},b_0,\dots,b_n)},
        \]
        where $a_n$ and $b_n$ are these $x_n$ and $y_n$.
        Let $\alpha=\sup_{y_n\in Y}(\alpha_{y_n}+1)$. As $\abs{Y}\leq\abs{\dense{\A}\cup\dense{\B}}=\kappa$, we get $\alpha<\kappa^+$. By the induction hypothesis,
        \[
            \playertwo\wins\EFD{\alpha+1,\varepsilon,(\varepsilon_0,\dots,\varepsilon_{n-1}),(0,\dots,n-1)}{\A,\B}{(\dense{\A},a_0,\dots,a_{n-1}),(\dense{\B},b_0,\dots,b_{n-1})}.
        \]
        Let, in this game, $\playerone$ play $x_n$, $\varepsilon_n$, $n$ and $\alpha$. Then the winning strategy gives us some $y_n\in Y$. But now
        \[
            \playertwo\wins\EFD{\alpha,\varepsilon,(\varepsilon_0,\dots,\varepsilon_n),(0,\dots,n)}{\A,\B}{(\dense{\A},a_0,\dots,a_n),(\dense{\B},b_0,\dots,b_n)},
        \]
        which, as $\alpha_{y_n}\leq\alpha$, implies
        \[
            \playertwo\wins\EFD{\alpha_{y_n},\varepsilon,(\varepsilon_0,\dots,\varepsilon_n),(0,\dots,n)}{\A,\B}{(\dense{\A},a_0,\dots,a_n),(\dense{\B},b_0,\dots,b_n)},
        \]
        a contradiction.
    \end{proof}

    \begin{corollary}
        \label{Corollary: I wins the long game between the dense sets if he wins it between the models.}
        If $\playerone\wins\EF{\omega,\varepsilon,\bar{\varepsilon}}{\A,\B}{(\A,a_0,\dots,a_{n-1}),(\B,b_0,\dots,b_{n-1})}$ and $\dense{\A}\subseteq\A$ and $\dense{\B}\subseteq\B$ are dense, then $\playerone\wins\EF{\omega,\varepsilon,\bar{\varepsilon}}{\A,\B}{(\dense{\A},a_0,\dots,a_{n-1}),(\dense{\B},b_0,\dots,b_{n-1})}$.
    \end{corollary}
    \begin{proof}
        If $\playerone$ wins the long game between the models, he wins, by Theorem~\ref{Theorem: If II wins many enough dynamic games then she wins the long game}, some dynamic game between the models, so by Corollary~\ref{Corollary: II wins in the dense sets if she wins in the models} he wins the same dynamic game between the dense sets, so by Theorem~\ref{Theorem: If II wins the long game she wins all dynamic games} he wins the long game between the dense sets.
    \end{proof}

    The games are neither symmetric nor transitive, but we have weak forms of symmetry and transitivity available to us.

    \begin{lemma}[Weak Symmetry]
        \label{Lemma: Symmetry}
        Let $\varepsilon>0$, $\bar{\varepsilon}\in(\varepsilon,\infty)^n$ and $\bar{k}\in\omega^n$. Then
        \[
            \playertwo\wins\EFD{\alpha,\varepsilon,\bar{\varepsilon},\bar{k}}{\A,\B}{(\dense\A,\bar{a}),(\dense\B,\bar{b})} \implies \playertwo\wins\EFD{\alpha,\varepsilon^+,\bar{\varepsilon}^+,\bar{k}^-}{\B,\A}{(\dense\B,\bar{b}),(\dense\A,\bar{a})},
        \]
        for any $\varepsilon^+>0$, $\bar{\varepsilon}^+\in(\varepsilon^+,\infty)^n$ and $\bar{k}^-\in\omega^n$ such that
        \begin{enumerate}
            \item $\varepsilon^+\geq\varepsilon$,
            \item $\varepsilon^+_i>\varepsilon_i$ for $i<n$ and
            \item $k^-_i\leq e^{-\varepsilon^+_i}k_i$ for $i<n$.
        \end{enumerate}
    \end{lemma}
    \begin{proof}
        Let $\tau$ be a winning strategy of $\playertwo$ in $\EFD{\alpha,\varepsilon,\bar{\varepsilon},\bar{k}}{\A,\B}{(\dense\A,\bar{a}),(\dense\B,\bar{b})}$. Then define a strategy $\tau'$ for $\playertwo$ in $\EFD{\alpha,\varepsilon^+,\bar{\varepsilon}^+,\bar{k}^-}{\B,\A}{(\dense\B,\bar{b}),(\dense\A,\bar{a})}$ by setting
        \[
            \tau'((x_j,\varepsilon^+_j,k^-_j)_{j\leq i}) = \tau((x_j,\varepsilon_j,k_j)_{j\leq i}),
        \]
        where $\varepsilon_j = \varepsilon^+ + (\varepsilon^+_j-\varepsilon^+)/2$ and $k_j=e^{\varepsilon^+_j}k^-_j$ for $n\leq j\leq i$. We show that $\tau'$ is winning. Let $((x_i,\varepsilon^+_i,k^-_i),y_i)_{i<l}$ be an arbitrary play where $\playertwo$ uses $\tau'$. Let $i<l$, $\varphi(v_0,\dots,v_{p-1})\in\good(k^-_i)$ and $j_0,\dots,j_{p-1}\geq k^-_i$, and suppose that
        \[
            \A\not\models\appr(\varphi,{\varepsilon^+_i})(a_{j_0},\dots,a_{j_{p-1}}).
        \]
        We need to show that
        \[
            \B\not\models\varphi(b_{j_0},\dots,b_{j_{p-1}}).
        \]
        As $\A\not\models\appr(\varphi,{\varepsilon^+_i})(a_{j_0},\dots,a_{j_{p-1}})$, we have $\A\models\weakneg(\appr(\varphi,{\varepsilon^+_i}))(a_{j_0},\dots,a_{j_{p-1}})$ by Definition~\ref{Definition: Vocabulary with approximations}~\ref{item: Weak negation}~\ref{item: Weak negation property}. Now, as $e^{\varepsilon^+_i}k^-_i\leq k_i$, we have $\weakneg(\appr(\varphi,{\varepsilon^+_i}))\in\good(k_i)$, so as $\tau$ is a winning strategy in the other game and
        \[
            \A\models\weakneg(\appr(\varphi,{\varepsilon^+_i}))(a_{j_0},\dots,a_{j_{p-1}}),
        \]
        we have
        \[
            \B\models\appr(\weakneg(\appr(\varphi,{\varepsilon^+_i})),{\varepsilon_i})(b_{j_0},\dots,b_{j_{p-1}}).
        \]
        Denote $\delta=\varepsilon^+_i-\varepsilon_i$. Now $\delta>0$ and $\appr(\weakneg(\appr(\varphi,{\varepsilon^+_i})),{\varepsilon_i}) = \weakneg(\appr(\varphi,{\delta}))$, so since $\B\models\weakneg(\appr(\varphi,{\delta}))(b_{j_0},\dots,b_{j_{p-1}})$, we have $\B\not\models\varphi(b_{j_0},\dots,b_{j_{p-1}})$ by Definition~\ref{Definition: Vocabulary with approximations}~\ref{item: Weak negation}~\ref{item: Weak negation property}.
    \end{proof}

    \begin{lemma}[Weak Transitivity]
        \label{Lemma: Transitivity}
        Let $\alpha\in\On$, $\varepsilon,\delta,\zeta>0$, $\bar{\varepsilon}\in(\varepsilon,\infty)^n$, $\bar{\delta}\in(\delta,\infty)^n$, $\bar{\zeta}\in(\zeta,\infty)^n$ and $\bar{k},\bar{K}\in\omega^n$, and suppose that
        \[
            \playertwo\wins\EFD{\alpha,\varepsilon,\bar{\varepsilon},\bar{k}}{\A,\B}{(\dense\A,\bar{a}),(\dense\B,\bar{b})} \quad\text{and}\quad \playertwo\wins\EFD{\alpha,\delta,\bar{\delta},\bar{K}}{\B,\model C}{(\dense\B,\bar{b}),(\dense{\model C},\bar{c})}.
        \]
        and
        \begin{enumerate}
            \item $\zeta\geq\varepsilon+\delta$, \label{item: transitivity-first}
            \item $\zeta_i\geq\varepsilon_i+\delta_i$, and
            \item $e^{\varepsilon_i}k_i\leq K_i$ \label{item: transitivity-last}
        \end{enumerate}
        for all $i<n$. Then
        \[
            \playertwo\wins\EFD{\alpha,\zeta,\bar{\zeta},\bar{k}}{\A,\model C}{(\dense\A,\bar{a}),(\dense{\model C},\bar{c})}.
        \]
    \end{lemma}
    \begin{proof}
        Denote
        \begin{align*}
            G_1 &= \EFD{\alpha,\varepsilon,\bar{\varepsilon},\bar{k}}{\A,\B}{(\dense\A,\bar{a}),(\dense\B,\bar{b})}, \\
            G_2 &= \EFD{\alpha,\delta,\bar{\delta},\bar{K}}{\B,\model C}{(\dense\B,\bar{b}),(\dense{\model C},\bar{c})} \text{ and} \\
            G_3 &= \EFD{\alpha,\zeta,\bar{\zeta},\bar{m}}{\A,\model C}{(\dense\A,\bar{a}),(\dense{\model C},\bar{c})}.
        \end{align*}
        Let $\tau_1$ be the strategy of $\playertwo$ in $G_1$ and $\tau_2$ her strategy in $G_2$. Then define a strategy $\tau_3$ in $G_3$ by setting
        \[
            \tau_3((x_j,\zeta_j,k_j)_{j\leq i}) \coloneqq \tau_2((z_j,\delta_j,K_j)_{j\leq i}),
        \]
        where
        \begin{itemize}
            \item $z_j = \tau_1((x_p,\varepsilon_p,k_p)_{p\leq j})$ for $j\leq i$,
            \item $\varepsilon_j = \frac{\varepsilon}{\zeta}\zeta_j$ and $\delta_j = \frac{\delta}{\zeta}\zeta_j$ for $n\leq j\leq i$, and
            \item $K_j = \ceil{e^{\varepsilon_j}k_j}$ for $n\leq j\leq i$.
        \end{itemize}
        Note that in a play of $G_3$ where $\playertwo$ uses $\tau_3$,~\ref{item: transitivity-first}--\ref{item: transitivity-last} hold for all $i$ (and not just for $i<n$).

        We show that $\tau_3$ is winning. Let $((x_i,\zeta_i,k_i),y_i)_{i < l}$ be an arbitrary play such that $y_i=\tau_3((x_j,\zeta_j,k_j)_{j\leq i})$ for all $i<l$. Let $i<l$, $\varphi(v_0,\dots,v_{p-1})\in\good(k_i)$ and $j_0,\dots,j_{p-1}\geq i$, and suppose that
        \[
            \A\models\varphi(a_{j_0},\dots,a_{j_{p-1}}).
        \]
        We need to show that
        \[
            \model{C}\models\appr(\varphi,{\zeta_i})(c_{j_0},\dots,c_{j_{p-1}}).
        \]
        Now, as $\tau_1$ is a winning strategy for $\playertwo$ in $G_1$ and $a_{j_0},\dots,a_{j_{p-1}}\in\{x_q, z_q \mid q<l \}$, we have
        \[
            \A\models\varphi(a_{j_0},\dots,a_{j_{p-1}}) \implies \B\models\appr(\varphi,{\varepsilon_i})(b_{j_0},\dots,b_{j_{p-1}}),
        \]
        whence $\B\models\appr(\varphi,{\varepsilon_i})(b_{j_0},\dots,b_{j_{p-1}})$. As $e^{\varepsilon_i}k_i\leq K_i$, we have $\appr(\varphi,{\varepsilon_i})\in\good(K_i)$. Now, as $\tau_2$ is a winning strategy for $\playertwo$ in $G_2$ and $b_{j_0},\dots,b_{j_{p-1}}\in\{z_q, y_q \mid q<l \}$, we have
        \[
            \B\models\appr(\varphi,{\varepsilon_i})(b_{j_0},\dots,b_{j_{p-1}}) \implies \model{C}\models\appr(\varphi,{\varepsilon_i+\delta_i})(c_{j_0},\dots,c_{j_{p-1}}),
        \]
        whence $\model{C}\models\appr(\varphi,{\varepsilon_i+\delta_i})(c_{j_0},\dots,c_{j_{p-1}})$. Notice that as
        \[
            \varepsilon_i+\delta_i = \zeta_i(\varepsilon+\delta)/\zeta \leq \zeta_i\zeta/\zeta = \zeta_i,
        \]
        we have $\appr(\varphi,{\varepsilon_i+\delta_i})\models\appr(\varphi,{\zeta_i})$, whence $\model{C}\models\appr(\varphi,{\zeta_i})(c_{j_0},\dots,c_{j_{p-1}})$.
    \end{proof}

    \subsection{Relation Games}

    Let $\corr=(\corr_\varepsilon)_{\varepsilon\geq 0}$ be a correspondence notion of $\class$. We define two more variants of the EF-game that are used to capture the concept of being in $\varepsilon$-correspondence.

    \begin{definition}
        Let $\A,\B\in\class$, and let $\dense\A$ and $\dense\B$ be a dense subset of the respective spaces. For $\varepsilon\geq 0$ and tuples $(a_0,\dots,a_{n-1})\in\A^n$ and $(b_0,\dots,b_{n-1})\in\B^n$, we denote by
        \[
            \EFR{\omega,\varepsilon}{\A,\B}{(\dense\A,a_0,\dots,a_{n-1}),(\dense\B,b_0,\dots,b_{n-1})}
        \]
        the game of length $\omega$ that is played between the sets $\dense\A$ and $\dense\B$ like the classic ordinary EF-game but with the following winning condition: $\playertwo$ wins if for all $i_0,\dots,i_{n-1}<\omega$ and all atomic $\varphi(v_0,\dots,v_{n-1})$ we have
        \[
            \A\models\varphi(a_{i_0},\dots,a_{i_{n-1}}) \implies \B\models\appr(\varphi,\varepsilon)(b_{i_0},\dots,b_{i_{n-1}}).
        \]

        The game
        \[
            \EFDR{\omega,\varepsilon,\alpha}{\A,\B}{(\dense\A,a_0,\dots,a_{n-1})),(\dense\B,b_0,\dots,b_{n-1})}
        \]
        is defined similarly but on each round $i$, $\playerone$ also chooses an ordinal $\alpha_i<\alpha$ so that $\alpha_{i+1}<\alpha_{i}$ for all $i$, and the game ends when $\playerone$ chooses $0$.
    \end{definition}

    \begin{theorem}
        For separable $\A,\B\in\class$, countable dense sets $\dense\A$ and $\dense\B$ of $\A$ and $\B$ respectively, and $\varepsilon>0$, the following are equivalent.
        \begin{enumerate}
            \item There is $\varepsilon'\in(0,\varepsilon)$ such that $\playertwo\wins\EFR{\omega,\varepsilon'}{\A,\B}{\dense\A,\dense\B}$.
            \item There is $\varepsilon'\in(0,\varepsilon)$ and an $\varepsilon'$-correspondence between $\A$ and $\B$ such that $\dom(R)=\dense\A$ and $\ran(R)=\dense\B$.
        \end{enumerate}
    \end{theorem}
    \begin{proof}
        Suppose that $\playertwo$ has a winning strategy in $\EFR{\omega,\varepsilon'}{\A,\B}{\dense\A,\dense\B}$. Let $R = \{ (a_i,b_i) \mid i<\omega \}$ be the relation generated by a play where $\playerone$ plays every element of $\dense\A\cup\dense\B$. But then $R$ is a set correspondence between $\dense\A$ and $\dense\B$ such that
        \[
            \A\models\varphi(a_{i_0},\dots,a_{i_{n-1}}) \implies \B\models\appr(\varphi,\varepsilon')(b_{i_0},\dots,b_{i_{n-1}})
        \]
        for all $(a_{i_0},b_{i_0}),\dots,(a_{i_{n-1}},b_{i_{n-1}})\in R$ and atomic $\varphi(v_0,\dots,v_{n-1})$, so the existence of an $\varepsilon''$-correspondence $R'$ with $\dom(R')=\dense\A$ and $\ran(R')=\dense\B$ for some $\varepsilon''\in(0,\varepsilon)$ follows by Definition~\ref{Definition: Correspondence notion}~\ref{item: Preserving atomic formulae suffices for epsilon-correspondence}.

        Then suppose that $R$ is an $\varepsilon'$-correspondence with $\dom(R)=\dense\A$ and $\ran(R)=\dense\B$. Then, by Definition~\ref{Definition: Correspondence notion}~\ref{item: Preserving atomic formulae suffices for epsilon-correspondence}, we can find another set correspondence $R'$ between $\dense\A$ and $\dense\B$ such that for some $\varepsilon''\in(0,\varepsilon)$ we have
        \[
            \A\models\varphi(\bar{a}) \implies \B\models\appr(\varphi,\varepsilon')(\bar{b})
        \]
        for all $(a_0,b_0),\dots,(a_{n-1},b_{n-1})\in R$ and atomic $\varphi(v_0,\dots,v_{n-1})$. But then
        $\playertwo$ wins the game $\EFR{\omega,\varepsilon''}{\A,\B}{\dense\A,\dense\B}$ by playing always moves that correspond to moves of $\playerone$ according to $R'$.
    \end{proof}

    The proofs of all the following lemmata are proved like in the function game setting (but are easier to prove).

    \begin{lemma}
        \label{Lemma: If II wins the long game she wins the dynamic games (relation games)}
        If $\playertwo\wins\EFR{\omega,\varepsilon}{\A,\B}{\dense\A,\dense\B}$, then $\playertwo\wins\EFR{\omega,\varepsilon,\alpha}{\A,\B}{\dense\A,\dense\B}$ for all $\alpha\in\On$.
    \end{lemma}

    \begin{lemma}
        \label{Lemma: If II wins all dynamic games she wins the long game (relation games)}
        If $\playertwo\wins\EFR{\omega,\varepsilon,\alpha}{\A,\B}{\dense\A,\dense\B}$ for all $\alpha<(\dense\A\cup\dense\B)^+$, then $\playertwo\wins\EFR{\omega,\varepsilon}{\A,\B}{\dense\A,\dense\B}$.
    \end{lemma}

    \begin{lemma}[Symmetry]
        \label{Lemma: Weak symmetry (relation games)}
        Let $\varepsilon>0$. Then
        \[
            \playertwo\wins\EFDR{\alpha,\varepsilon}{\A,\B}{(\dense\A,\bar{a}),(\dense\B,\bar{b})} \implies \playertwo\wins\EFDR{\alpha,\varepsilon}{\B,\A}{(\dense\B,\bar{b}),(\dense\A,\bar{a})}.
        \]
    \end{lemma}

    \begin{lemma}[Weak Transitivity]
        \label{Lemma: Weak transitivity (relation games)}
        Let $\varepsilon,\delta>0$ and suppose that
        \[
            \playertwo\wins\EFDR{\alpha,\varepsilon}{\A,\B}{(\dense\A,\bar{a}),(\dense\B,\bar{b})} \quad\text{and}\quad \playertwo\wins\EFDR{\alpha,\delta}{\B,\model C}{(\dense\B,\bar{b}),(\dense{\model C},\bar{c})}.
        \]
        Then
        \[
            \playertwo\wins\EFDR{\alpha,\varepsilon+\delta}{\A,\model C}{(\dense\A,\bar{a}),(\dense{\model C},\bar{c})}.
        \]
    \end{lemma}

    \section{Scott Rank}\label{Section: Scott Rank}
    
    \subsection{Scott rank for function games}

    \begin{definition}
        Let $\varepsilon\geq 0$, $\bar{\varepsilon}\in\open{\varepsilon,\infty}^n$ and $\bar{k}=(k_0,\dots,k_{n-1})\in\omega^n$, with $k_0<\dots<k_{n-1}$. For $\A,\B\in\class$ and tuples $\bar{a}\in\A^n$ and $\bar{b}\in\B^n$, we define the $\varepsilon$-Scott watershed $\watershed{\varepsilon}{\bar{\varepsilon},\bar{k}}{(\A,\bar{a}),(\B,\bar{b})}$ of the pair $((\A,\bar{a}),(\B,\bar{b}))$ with parameters $\bar{\varepsilon}$ and $\bar{k}$ to be the least ordinal $\alpha$ such that $\playerone\wins\EFD{\alpha,\varepsilon,\bar{\varepsilon},\bar{k}}{\A,\B}{(\A,\bar{a}),(\B,\bar{b})}$ if such ordinal exists and $0$ otherwise.
    \end{definition}

    \begin{remark}
        Since the dynamic EF-game is determined, the watershed always exists unless $\playertwo$ wins the game for all $\alpha$. Note that it is also a successor ordinal, as $\playertwo$ having a winning strategy for every ordinal below some limit implies $\playertwo$ having a winning strategy at the limit.
    \end{remark}

    \begin{definition}
        The $\varepsilon$-Scott rank $\scottrank_\varepsilon(\A)$ of $\A\in\class$ is the supremum of all ordinals
        \[
            \sup_{n<\omega}\sup_{\bar{\varepsilon}\in\open{\varepsilon,\infty}^n}\sup_{\bar{k}\in\omega^n}\sup_{\bar{a}\in\A^n}\sup_{\bar{b}\in\B^n}\watershed{\varepsilon}{\bar{\varepsilon},\bar{k}}{(\A,\bar{a}),(\B,\bar{b})}
        \]
        over all $\B\in\class$ such that $\density(\B)\leq\density(\A)$.
    \end{definition}

    \begin{lemma}
        \label{Lemma: Upper bound for Scott rank}
        Let $\kappa$ be an infinite cardinal with $\density(\A)\leq\kappa$. Then $\scottrank_\varepsilon(\A)\leq\kappa^+$.
    \end{lemma}
    \begin{proof}
        By Theorem~\ref{Theorem: If II wins many enough dynamic games then she wins the long game}, if $\playerone$ wins any dynamic game between $\A$ and some other space $\B$ with $\density(\B)\leq\density(\A)$, he will win one with the clocking ordinal below $\kappa^+$. Thus $\scottrank_\varepsilon(\A)$ is a supremum of ordinals $<\kappa^+$ and hence $\scottrank_\varepsilon(\A)\leq\kappa^+$.
    \end{proof}

    \begin{lemma}
        \label{Lemma: Properties of Scott rank}
        $\scottrank_\varepsilon(\A)$ is the least ordinal $\alpha$ such that for any other $\B\in\class$ such that $\density(\B)\leq\density(\A)$ and for all $n<\omega$, $\bar{a}\in\A^n$, $\bar{b}\in\B^n$, $\bar{\varepsilon}\in\open{\varepsilon,\infty}^n$ and $\bar{k}\in\omega^n$, if
        \[
            \playertwo\wins\EFD{\alpha,\varepsilon,\bar{\varepsilon},\bar{k}}{\A,\B}{(\A,\bar{a}),(\B,\bar{b})},
        \]
        then
        \[
            \playertwo\wins\EFD{\alpha+1,\varepsilon,\bar{\varepsilon},\bar{k}}{\A,\B}{(\A,\bar{a}),(\B,\bar{b})}.
        \]
    \end{lemma}
    \begin{proof}
        An adaptation of a proof for the corresponding lemma in the classical Scott rank setting works.
    \end{proof}

    \begin{question}
        By Lemma~\ref{Lemma: Upper bound for Scott rank} the $\varepsilon$-Scott rank of a separable model is at most $\omega_1$. Can this be improved or do there exist a separable $\A\in\class$ and $\varepsilon\geq 0$ with $\scottrank_\varepsilon(\A)\geq\omega_1$?
    \end{question}

    For classical (discrete) structures, the Scott rank of a countable model is always countable. This is essentially because the corresponding dynamic game is transitive, which guarantees that it is enough to consider the watershed of pairs
    \[
        ((\A,\bar{a}),(\A,\bar{b}))
    \]
    whereas here we seem to need to go through all ($\varepsilon$-isomorphic) copies $\B$ of $\A$. As our games are not transitive, we do not have an obvious way of utilizing Scott rank defined only inside the space $\A$. However, if we are only interested in situations where $\playertwo$ wins all $\varepsilon$-games for $\varepsilon>0$, we can use weak transitivity (Lemma~\ref{Lemma: Transitivity}) to obtain some similar results to the classical ones.

    \begin{definition}
        We denote by $\innerscottrank(\A)$ the ordinal
        \[
            \sup_{\varepsilon(0,\infty)\cap D}\sup_{n<\omega}\sup_{\bar{\varepsilon}\in\open{\varepsilon,\infty}^n\cap\Q}\sup_{\bar{k}\in\omega^n}\sup_{\bar{a}\in\dense\A^n}\sup_{\bar{b}\in\dense\A^n}\watershed{\varepsilon}{\bar{\varepsilon},\bar{k}}{(\dense\A,\bar{a}),(\dense\A,\bar{b})},
        \]
        where $\dense\A\subseteq\A$ is a dense set of minimal cardinality.
    \end{definition}

    \begin{remark}
        Note that since a player who has a winning strategy in the dynamic game can always choose to play inside any dense sets, the definition of $\innerscottrank(\A)$ does not depend on the set $\dense\A$.
    \end{remark}
    \begin{remark}
        Note that $\innerscottrank(\A)<\density(\A)^+$.
    \end{remark}

    \begin{lemma}
        \label{Lemma: Inner Scott Rank property}
        $\innerscottrank(\A)$ is the least ordinal $\alpha$ such that for all $\varepsilon>0$, $n<\omega$, $\bar{\varepsilon}\in(\varepsilon,\infty)^n$, $\bar{k}\in\omega^n$, $\bar{a}\in\A^n$ and $\bar{b}\in\A^n$ if
        \[
            \playertwo\wins\EFD{\alpha,\varepsilon,\bar{\varepsilon},\bar{k}}{\A,\A}{(\dense\A,\bar{a}),(\dense\A,\bar{b})},
        \]
        then
        \[
            \playertwo\wins\EFD{\alpha+1,\varepsilon,\bar{\varepsilon},\bar{k}}{\A,\A}{(\dense\A,\bar{a}),(\dense\A,\bar{b})}.
        \]
    \end{lemma}
    \begin{proof}
        Similar to the proof of Lemma~\ref{Lemma: Properties of Scott rank}.
    \end{proof}

    \begin{theorem}
        \label{Theorem: Inner Scott Rank works as intended}
        Let $\alpha=\innerscottrank(\A)$. Suppose that for each $\varepsilon>0$,
        \[
            \playertwo\wins\EFD{\alpha+\omega,\varepsilon}{\A,\B}{\dense\A,\dense\B}.
        \]
        Then for all $\varepsilon>0$, $n<\omega$, $\bar{\varepsilon}\in(\varepsilon,\infty)^n$, $\bar{k}\in\omega^n$, $\bar{a}\in\A^n$ and $\bar{b}\in\B^n$ if
        \[
            \playertwo\wins\EFD{\alpha,\varepsilon,\bar{\varepsilon},\bar{k}}{\A,\B}{(\dense\A,\bar{a}),(\dense\B,\bar{b})},
        \]
        then
        \[
            \playertwo\wins\EFD{\alpha+1,\varepsilon^+,\bar{\varepsilon}^+,\bar{k}}{\A,\B}{(\dense\A,\bar{a}),(\dense\B,\bar{b})}
        \]
        for any $\varepsilon^+>\varepsilon$ and $\varepsilon^+_i>\varepsilon_i$, $i<n$.
    \end{theorem}
    \begin{proof}
        First note that from the assumption that
        \begin{equation}
            \playertwo\wins\EFD{\alpha+\omega,\varepsilon}{\A,\B}{\dense\A,\dense\B} \quad\text{for all $\varepsilon>0$}
            \label{eq: II wins the initial game}
        \end{equation}
        it follows by Lemma~\ref{Lemma: Symmetry} that we also have
        \begin{equation}
            \playertwo\wins\EFD{\alpha+\omega,\varepsilon}{\B,\A}{\dense\B,\dense\A} \quad\text{for all $\varepsilon>0$}.
            \label{eq: II wins the initial reversed game}
        \end{equation}
        Let $\varepsilon>0$, $n<\omega$, $\bar{\varepsilon}\in(\varepsilon,\infty)^n$, $\bar{k}\in\omega^n$, $\bar{a}\in\A^n$ and $\bar{b}\in\B^n$, and suppose that
        \begin{equation}
            \playertwo\wins\EFD{\alpha,\varepsilon,\bar{\varepsilon},\bar{k}}{\A,\B}{(\dense\A,\bar{a}),(\dense\B,\bar{b})}.
            \label{eq: free round assumption}
        \end{equation}
        We need to show that
        \begin{equation}
            \playertwo\wins\EFD{\alpha+1,\varepsilon^+,\bar{\varepsilon}^+,\bar{k}}{\A,\B}{(\dense\A,\bar{a}),(\dense\B,\bar{b})}
            \label{eq: free round conclusion}
        \end{equation}
        holds for all $\varepsilon^+>\varepsilon$ and $\varepsilon^+_i>\varepsilon_i$, $i<n$. So fix $\varepsilon^+$ and $\varepsilon^+_i$, and let
        \begin{align*}
            \delta   &= (\varepsilon^+-\varepsilon)/3, \\
            \delta_i &= (\varepsilon^+_i-\varepsilon_i)/3 \quad\text{and}\\
            K_i     &= \ceil{e^{\varepsilon_i+3\delta_i}k_i + e^{2\delta_i}}.
        \end{align*}
        We now play $n$ rounds of the game $\EFD{\alpha+n+1,\delta}{\B,\A}{\dense\B,\dense\A}$, where $\playertwo$ has a winning strategy by~\eqref{eq: II wins the initial reversed game}. We let $\playerone$ play $(b_0,\delta_0,K_0),\dots,(b_{n-1},\delta_{n-1},K_{n-1})$, and $\playertwo$ responds with some $c_0,\dots,c_{n-1}\in\dense\A$ following her winning strategy, so we get
        \begin{equation}
            \playertwo\wins\EFD{\alpha+1,\delta,\bar{\delta},\bar{K}}{\B,\A}{(\dense\B,b_0,\dots,b_{n-1}),(\dense\A,c_0,\dots,c_{n-1})}.
            \label{eq: Game from b to c}
        \end{equation}
        Then, as $K_i\geq e^{\varepsilon_i}k_i$, by using Lemma~\ref{Lemma: Transitivity} on~\eqref{eq: free round assumption} and~\eqref{eq: Game from b to c} we obtain
        \[
            \playertwo\wins\EFD{\alpha,\zeta,\bar{\zeta},\bar{k}}{\A,\A}{(\dense\A,a_0,\dots,a_{n-1}),(\dense\A,c_0,\dots,c_{n-1})},
        \]
        where $\zeta = \varepsilon+\delta$ and $\zeta_i = \varepsilon_i + \delta_i$. Now by Lemma~\ref{Lemma: Inner Scott Rank property}, we have
        \begin{equation}
            \playertwo\wins\EFD{\alpha+1,\zeta,\bar{\zeta},\bar{k}}{\A,\A}{(\dense\A,a_0,\dots,a_{n-1}),(\dense\A,c_0,\dots,c_{n-1})}.
            \label{eq: Game from a to c}
        \end{equation}
        Now using Lemma~\ref{Lemma: Symmetry} on~\eqref{eq: Game from b to c} we get
        \begin{equation}
            \playertwo\wins\EFD{\alpha+1,2\delta,\bar{2\delta},\bar{K}^-}{\A,\B}{(\dense\A,c_0,\dots,c_{n-1}),(\dense\B,b_0,\dots,b_{n-1})},
            \label{eq: Game from c to b}
        \end{equation}
        where $K^-_i = \floor{e^{-2\delta_i}K_i}$. Now, as $\zeta + 2\delta = \varepsilon^+$ and $\zeta_i+2\delta_i = \varepsilon^+_i$, and $K^-_i\geq e^{\zeta_i}k_i$, we may apply Lemma~\ref{Lemma: Transitivity} on~\eqref{eq: Game from a to c} and~\eqref{eq: Game from c to b} and finally get~\eqref{eq: free round conclusion}.
    \end{proof}

    \begin{theorem}
        \label{Theorem: If II wins the Inner Scott rank + omega game she wins the long game}
        Let $\alpha=\innerscottrank(\A)$. Suppose that for each $\varepsilon>0$,
        \[
            \playertwo\wins\EFD{\alpha+\omega,\varepsilon}{\A,\B}{\dense\A,\dense\B}.
        \]
        Then for each $\varepsilon>0$,
        \[
            \playertwo\wins\EF{\omega,\varepsilon}{\A,\B}{\dense\A,\dense\B}.
        \]
    \end{theorem}
    \begin{proof}
        Let $\varepsilon>0$. We show that $\playertwo$ wins $G=\EF{\omega,\varepsilon}{\A,\B}{\dense\A,\dense\B}$. The strategy of $\playertwo$ is to play such moves $y_n$ that the following condition remains true:
        \begin{itemize}
            \item[$(*)$] if the position in $G$ is $((x_i,\varepsilon_i),y_i)_{i<n}$, then
            \[
                \playertwo\wins\EFD{\alpha,\delta^n,(\delta^n_0,\dots,\delta^n_{n-1}),(0,\dots,n-1)}{\A,\B}{(\dense\A,a_0,\dots,a_{n-1}),(\dense\B,b_0,\dots,b_{n-1})},
            \]
            where $\delta^n = \sum_{j<n+1}\varepsilon/2^{j+1}$ and $\delta^n_i = \sum_{j<n-i}\varepsilon_i/2^{j+1}$ for $i<n$.
        \end{itemize}
        It is easy to see that if $\playertwo$ can play such moves, she wins the game $G$. So we show that such moves can always be played.

        In the beginning of $G$, $(*)$ holds by the assumption that $\playertwo\wins\EFD{\alpha+\omega,\varepsilon/2}{\A,\B}{\dense\A,\dense\B}$. Suppose that we have played $n$ rounds, the current position is $((x_i,\varepsilon_i),y_i)_{i<n}$ and $\playertwo$ has been able to maintain $(*)$, i.e.
        \begin{equation}
            \playertwo\wins\EFD{\alpha,\delta^n,(\delta^n_0,\dots,\delta^n_{n-1}),(0,\dots,n-1)}{\A,\B}{(\dense\A,a_0,\dots,a_{n-1}),(\dense\B,b_0,\dots,b_{n-1})}
            \label{eq: induction hypothesis}
        \end{equation}
        holds. Suppose that $\playerone$ now plays $x_n$ and $\varepsilon_n$. As $\delta^n<\delta^{n+1}$ and $\delta^n_i<\delta^{n+1}_i$ for $i<n$, we may apply Theorem~\ref{Theorem: Inner Scott Rank works as intended} on~\eqref{eq: induction hypothesis} and obtain
        \[
            \playertwo\wins\EFD{\alpha+1,\delta^{n+1},(\delta^{n+1}_0,\dots,\delta^{n+1}_{n-1}),(0,\dots,n-1)}{\A,\B}{(\dense\A,a_0,\dots,a_{n-1}),(\dense\B,b_0,\dots,b_{n-1})}.
        \]
        Then we play the above game and let $\playerone$ play $(x_n,\delta^{n+1}_n,n)$. Then we let $y_n$ be the response of $\playertwo$, according to her winning strategy in this game. But then $y_n$ is such that
        \[
            \playertwo\wins\EFD{\alpha,\delta^{n+1},(\delta^{n+1}_0,\dots,\delta^{n+1}_{n}),(0,\dots,n)}{\A,\B}{(\dense\A,a_0,\dots,a_{n}),(\dense\B,b_0,\dots,b_{n})},
        \]
        so it suffices as the response of $\playertwo$ in the game $G$.
    \end{proof}

    \subsection{Scott rank for relation games}

    \begin{definition}
        Let $\varepsilon>0$ and $\A,\B\in\class$. For dense sets $\dense\A\subseteq\A$ and $\dense\B\subseteq\B$ and tuples $\bar{a}\in\A^n$ and $\bar{b}\in\B^n$, we define the $\varepsilon$-Scott watershed $\watershedr{\varepsilon}{(\dense\A,\bar{a}),(\dense\B,\bar{b})}$ of the pair $((\dense\A,\bar{a}),(\dense\B,\bar{b}))$ to be the least ordinal $\alpha$ such that $\playerone\wins\EFDR{\alpha,\varepsilon}{\A,\B}{(\dense\A,\bar{a}),(\dense\B,\bar{b})}$ if such an ordinal exists and $0$ otherwise.
    \end{definition}

    \begin{definition}
        The $\varepsilon$-Scott rank $\scottrankr_\varepsilon(\A)$ of $\A\in\class$ is the supremum of the ordinals
        \[
            \sup_{n<\omega}\sup_{\bar{a}\in\A^n}\sup_{\bar{b}\in\B^n}\watershedr{\varepsilon}{(\dense\A,\bar{a}),(\B,\bar{b})}
        \]
        over all $\B\in\class$ such that $\density(\B)\leq\density(\A)$.
    \end{definition}

    \begin{definition}
        We denote by $\innerscottrankr(\A)$ the ordinal
        \[
            \sup_{\varepsilon\in(0,\infty)\cap D}\sup_{n<\omega}\sup_{\bar{a}\in\dense\A^n}\sup_{\bar{b}\in\dense\A^n}\watershedr{\varepsilon}{(\dense\A,\bar{a}),(\dense\A,\bar{b})},
        \]
        where $\dense\A\subseteq\A$ is a dense set of minimal cardinality.
    \end{definition}
    
    \begin{remark}
        Note that $\innerscottrankr(\A)<\density(\A)^+$.
    \end{remark}

    The following are proved in a fashion similar to their function game counterparts.

    \begin{lemma}
        $\scottrankr_\varepsilon(\A)$ is the least ordinal $\alpha$ such that for any other $\B\in\class$ such that $\density(\B)\leq\density(\A)$ and for all $n<\omega$, $\bar{a}\in\A^n$ and $\bar{b}\in\B^n$, if
        \[
            \playertwo\wins\EFDR{\alpha,\varepsilon}{\A,\B}{(\dense\A,\bar{a}),(\B,\bar{b})},
        \]
        then
        \[
            \playertwo\wins\EFDR{\alpha+1,\varepsilon}{\A,\B}{(\dense\A,\bar{a}),(\B,\bar{b})},
        \]
    \end{lemma}

    \begin{theorem}
        Let $\alpha = \innerscottrankr(\A)$. Suppose that for each $\varepsilon>0$,
        \[
            \playertwo\wins\EFDR{\alpha+\omega,\varepsilon}{\A,\B}{\dense\A,\dense\B}.
        \]
        Then for all $\varepsilon>0$, $n<\omega$, $\bar{a}\in\A^n$ and $\bar{b}\in\B^n$, if
        \[
            \playertwo\wins\EFDR{\alpha,\varepsilon}{\A,\B}{(\dense\A,\bar{a}),(\dense\B,\bar{b})},
        \]
        then
        \[
            \playertwo\wins\EFDR{\alpha+1,\varepsilon^+}{\A,\B}{(\dense\A,\bar{a}),(\dense\B,\bar{b})},
        \]
        for any $\varepsilon^+>\varepsilon$.
    \end{theorem}

    \section{Scott Sentences}\label{Section: Scott sentences}

    \subsection{Scott sentences for isomorphism notions}

    \begin{definition}
        Let $\A\in\class$, $\varepsilon\geq 0$ and $\dense{\A}\subseteq\A$ a dense set of cardinality $<\kappa$. Given $n<\omega$, $\bar{\varepsilon}\in\open{\varepsilon,\infty}^n$ $\bar{k}\in\omega^n$ and $\bar{a}\in\A^n$, define the formulae $\gameformula{\varepsilon}{\alpha,\bar{\varepsilon},\bar{k}}{\A,\bar{a}}(v_0,\dots,v_{n-1})\in\Logic\kappa$, $\alpha\in\On$, by induction on $\alpha$ as follows:
        \begin{enumerate}
            \item For $\alpha=0$,
            \begin{align*}
                \gameformula{\varepsilon}{\alpha,\bar{\varepsilon},\bar{k}}{\A,\bar{a}}(v_0,\dots,v_{n-1}) = \bigwedge\{&\appr(\varphi,{\varepsilon_i})(v_{j_0},\dots,v_{j_{k_i-1}}) \mid  i<n,\ \varphi\in\good(k_i)\\
                &j_0,\dots,j_{k_i-1}\geq i,\ \A\models\varphi(a_{j_0},\dots,a_{j_{k_i-1}}) \},
            \end{align*}
            \item for $\alpha=\beta+1$,
            \begin{align*}
                \gameformula{\varepsilon}{\alpha,\bar{\varepsilon},\bar{k}}{\A,\bar{a}}(v_0,\dots,v_{n-1}) =& \bigwedge_{\varepsilon_n\in\open{\varepsilon,\infty}\cap D}\bigwedge_{k_n<\omega}\left( \bigwedge_{a_n\in\dense{\A}}\exists v_n \gameformula{\varepsilon}{\beta,\concat{\bar{\varepsilon}}{\varepsilon_n},\concat{\bar{k}}{k_n}}{\A,\concat{\bar{a}}{a_n}}(v_0,\dots,v_n)\right. \\
                &\left.{}\land\forall v_n\bigvee_{a_n\in\dense{\A}}\gameformula{\varepsilon}{\beta,\concat{\bar{\varepsilon}}{\varepsilon_n},\concat{\bar{k}}{k_n}}{\A,\concat{\bar{a}}{a_n}}(v_0,\dots,v_n) \right),
            \end{align*}
            and
            \item for $\alpha$ limit,
            \[
                \gameformula{\varepsilon}{\alpha,\bar{\varepsilon},\bar{k}}{\A,\bar{a}}(v_0,\dots,v_{n-1})=\bigwedge_{\beta<\alpha}\gameformula{\varepsilon}{\beta,\bar{\varepsilon},\bar{k}}{\A,\bar{a}}(v_0,\dots,v_{n-1}).
            \]
        \end{enumerate}
        Whenever $n=0$, we leave out $\bar{\varepsilon}$, $\bar{k}$ and $\bar{a}$ (who all are the empty tuple) from the notation.
    \end{definition}

    \begin{lemma}
        \label{formulae vs games}
        For $\A,\B\in\class$, and $\bar{a}\in\A^n$ and $\bar{b}\in\B^n$, the following are equivalent:
        \begin{enumerate}
            \item $\B\models\gameformula{\varepsilon}{\alpha,\bar{\varepsilon},\bar{k}}{\A,\bar{a}}(\bar{b})$,
            \item $\playertwo\wins\EFD{\alpha,\varepsilon,\bar{\varepsilon},\bar{k}}{\A,\B}{(\dense{\A},\bar{a}),(\B,\bar{b})}$.
        \end{enumerate}
    \end{lemma}
    \begin{proof}
        Follows from the definitions by induction on $\alpha$.
    \end{proof}

    \begin{remark}
        By Corollary~\ref{Corollary: II wins in the dense sets iff she wins in the models}, if $\dense{\A}$ and $\dense{\A}'$ are two different dense sets in $\A$, we get
        \begin{align*}
            \playertwo\wins\EFD{\alpha,\varepsilon,\bar{\varepsilon},\bar{k}}{\A,\B}{(\dense{\A},\bar{a}),(\B,\bar{b})} &\iff \playertwo\wins\EFD{\alpha,\varepsilon,\bar{\varepsilon},\bar{k}}{\A,\B}{(\A,\bar{a}),(\B,\bar{b})} \\
            &\iff \playertwo\wins\EFD{\alpha,\varepsilon,\bar{\varepsilon},\bar{k}}{\A,\B}{(\dense{\A}',\bar{a}),(\B,\bar{b})},
        \end{align*}
        so constructing $\gameformula{\varepsilon}{\alpha,\bar{\varepsilon},\bar{k}}{\A,\bar{a}}(\bar{x})$ using two different dense sets of cardinality $\kappa$ gives logically equivalent formulae.
    \end{remark}

    \begin{definition}
        The $\varepsilon$-Scott sentence $\sigma^\A_\varepsilon$ of $\A\in\class$ is the sentence
        \[
            \gameformula{\varepsilon}{\scottrank_\varepsilon(\A)}{\A} \land \bigwedge_{n<\omega}\bigwedge_{\bar{\varepsilon}\in(\open{\varepsilon,\infty}\cap\Q)^n}\bigwedge_{\bar{k}\in\omega^n}\bigwedge_{\bar{a}\in\dense{\A}^n}\forall v_0\dots\forall v_{n-1}\left( \gameformula{\varepsilon}{\scottrank_\varepsilon(\A),\bar{\varepsilon},\bar{k}}{\A,\bar{a}} \to \gameformula{\varepsilon}{\scottrank_\varepsilon(\A)+1,\bar{\varepsilon},\bar{k}}{\A,\bar{a}} \right).
        \]
    \end{definition}

    \begin{theorem}
        \label{Theorem: Scott sentence guarantees win for II in the long game}
        For $\A,\B\in\class$, the following are equivalent:
        \begin{enumerate}
            \item $\B\models\sigma^\A_\varepsilon$,
            \item $\playertwo\wins\EF{\omega,\varepsilon}{\A,\B}{\A,\B}$.
        \end{enumerate}
    \end{theorem}
    \begin{proof}
        First assume $\B\models\sigma^\A_\varepsilon$. By Corollary~\ref{Corollary: I wins the long game between the dense sets if he wins it between the models.}, it is enough for $\playertwo$ to win the game $\EF{\omega,\varepsilon}{\A,\B}{\dense{\A},\B}$. The strategy of $\playertwo$ in this game is to play such moves $y_i$ that at all times the following holds:
        \begin{itemize}
            \item[$(*)$] If the position in $\EF{\omega,\varepsilon}{\A,\B}{\dense{\A},\B}$ is $((x_0,\varepsilon_0),y_0,\dots,(x_{n-1},\varepsilon_{n-1}),y_{n-1})$, then $\B\models\gameformula{\varepsilon}{\scottrank_\varepsilon(\A),(\varepsilon_0,\dots,\varepsilon_{n-1}),(0,\dots,n-1)}{\A,(a_0,\dots,a_{n-1})}(b_0,\dots,b_{n-1})$.
        \end{itemize}
        This strategy guarantees the win in the long game, as the same argument that we used in the proof of Theorem~\ref{Theorem: If II wins many enough dynamic games then she wins the long game} works also here (after the observations that is Lemma~\ref{formulae vs games}).

        So we show that there is such a strategy. In the beginning when no moves have been made, we know that $\B\models\gameformula{\varepsilon}{\scottrank_\varepsilon(\A)}{\A}$ by the definition of the Scott sentence. So suppose we are at round $n$ and the position is $((x_0,\varepsilon_0),y_0,\dots,(x_{n-1},\varepsilon_{n-1}),y_{n-1})$. We let $\playerone$ play $x_n$ and $\varepsilon_n$. Assume that $x_n\in\B$ (the other case is similar). By the induction hypothesis,
        \[
            \B\models\gameformula{\varepsilon}{\scottrank_\varepsilon(\A),(\varepsilon_0,\dots,\varepsilon_{n-1}),(0,\dots,n-1)}{\A,(a_0,\dots,a_{n-1})}(b_0,\dots,b_{n-1}).
        \]
        But then from the Scott sentence it follows that
        \[
            \B\models\gameformula{\varepsilon}{\scottrank_\varepsilon(\A)+1,(\varepsilon_0,\dots,\varepsilon_{n-1}),(0,\dots,n-1)}{\A,(a_0,\dots,a_{n-1})}(b_0,\dots,b_{n-1}).
        \]
        Thus
        \[
            \B\models\bigwedge_{\varepsilon_n}\bigwedge_{k_n}\forall v_n\bigvee_{a_n}\gameformula{\varepsilon}{\scottrank_\varepsilon(\A),(\varepsilon_0,\dots,\varepsilon_{n}),(0,\dots,n)}{\A,(a_0,\dots,a_{n})}(b_0,\dots,b_{n-1},v_n),
        \]
        so we can find some $y_n\in\dense{\A}$ such that $\B\models\gameformula{\varepsilon}{\scottrank_\varepsilon(\A),(\varepsilon_0,\dots,\varepsilon_{n}),(0,\dots,n)}{\A,(a_0,\dots,a_{n})}(b_0,\dots,b_{n})$.\footnote{Notice that while $\playerone$ plays any $\varepsilon_n$, the conjunction is taken over rational $\varepsilon_n$. This, however, does not matter as one can pick a slightly smaller rational number and choose $y_n$ according to this stricter $\varepsilon_n$ and this will not be a worse choice for $\playertwo$.} This completes this direction of the proof.

        Then assume that $\playertwo\wins\EF{\omega,\varepsilon}{\A,\B}{\A,\B}$. By Lemma~\ref{Theorem: If II wins the long game she wins all dynamic games}, $\playertwo\wins\EFD{\alpha,\varepsilon}{\A,\B}{\A,\B}$ for all $\alpha\in\On$. In particular, she wins the game $\EFD{\scottrank_\varepsilon(\A),\varepsilon}{\A,\B}{\A,\B}$. By Lemma~\ref{formulae vs games}, $\B\models\gameformula{\varepsilon}{\scottrank_\varepsilon(\A)}{\A}$. Left is to show that
        \[
            \B\models\bigwedge_{n<\omega}\bigwedge_{\bar{\varepsilon}\in(\open{\varepsilon,\infty}\cap\Q)^n}\bigwedge_{\bar{k}\in\omega^n}\bigwedge_{\bar{a}\in\dense{\A}^n}\forall v_0\dots\forall v_{n-1}\left( \gameformula{\varepsilon}{\scottrank_\varepsilon(\A),\bar{\varepsilon},\bar{k}}{\A,\bar{a}} \to \gameformula{\varepsilon}{\scottrank_\varepsilon(\A)+1,\bar{\varepsilon},\bar{k}}{\A,\bar{a}} \right).
        \]
        So let $n$, $\bar{\varepsilon}$, $\bar{k}$, and $\bar{a}\in\dense\A^n$ and $\bar{b}\in\B^n$ be arbitrary, and suppose that $\B\models\gameformula{\varepsilon}{\scottrank_\varepsilon(\A),\bar{\varepsilon},\bar{k}}{\A,\bar{a}}(\bar{b})$. Again by Lemma~\ref{formulae vs games}, this means that
        \[
            \playertwo\wins\EFD{\scottrank_\varepsilon(\A),\varepsilon,\bar{\varepsilon},\bar{k}}{\A,\B}{(\dense{\A},\bar{a}),(\B,\bar{b})}.
        \]
        By Lemma~\ref{Lemma: Properties of Scott rank}, we get that
        \[
            \playertwo\wins\EFD{\scottrank_\varepsilon(\A)+1,\varepsilon,\bar{\varepsilon},\bar{k}}{\A,\B}{(\dense{\A},\bar{a}),(\B,\bar{b})}.
        \]
        Again, by Lemma~\ref{formulae vs games}, we get that $\B\models\gameformula{\varepsilon}{\scottrank_\varepsilon(\A)+1,\bar{\varepsilon},\bar{k}}{\A,\bar{a}}(\bar{b})$, showing that
        \[
            \B\models\gameformula{\varepsilon}{\scottrank_\varepsilon(\A),\bar{\varepsilon},\bar{k}}{\A,\bar{a}}(\bar{b})\to\gameformula{\varepsilon}{\scottrank_\varepsilon(\A)+1,\bar{\varepsilon},\bar{k}}{\A,\bar{a}}(\bar{b}).
        \]
        As $\bar{b}$ was arbitrary, we get that
        \[
            \B\models\forall v_0\dots\forall v_{n-1}\left( \gameformula{\varepsilon}{\scottrank_\varepsilon(\A),\bar{\varepsilon},\bar{k}}{\A,\bar{a}} \to \gameformula{\varepsilon}{\scottrank_\varepsilon(\A)+1,\bar{\varepsilon},\bar{k}}{\A,\bar{a}} \right),
        \]
        and since all the other parameters were also arbitrary, we can conclude that
        \[
            \B\models\bigwedge_{n<\omega}\bigwedge_{\bar{\varepsilon}\in(\open{\varepsilon,\infty}\cap\Q)^n}\bigwedge_{\bar{k}\in\omega^n}\bigwedge_{\bar{a}\in\dense{\A}^n}\forall v_0\dots\forall v_{n-1}\left( \gameformula{\varepsilon}{\scottrank_\varepsilon(\A),\bar{\varepsilon},\bar{k}}{\A,\bar{a}} \to \gameformula{\varepsilon}{\scottrank_\varepsilon(\A)+1,\bar{\varepsilon},\bar{k}}{\A,\bar{a}} \right),
        \]
        finishing the proof.
    \end{proof}

    \begin{corollary}
        For separable $\A,\B\in\class$, $\B\models\sigma^\A_\varepsilon$ if and only if there exists an $\varepsilon$-isomorphism $\A\to\B$.
    \end{corollary}
    \begin{proof}
        Combine Theorems~\ref{Theorem: If II wins the long game then there exists an epsilon-isomorphism} and~\ref{Theorem: Scott sentence guarantees win for II in the long game}.
    \end{proof}

    \begin{remark}
        Let $\class$ be a class of metric structures and $\isom = (\isom_{\varepsilon})_{\varepsilon\geq 0}$ an isomorphism notion of $\class$. Then, if $L$ is a good vocabulary for $(\class,(\isom_{\varepsilon})_{\varepsilon\geq 0})$, we have
        \[
            d_{\isom}(\A,\B) = \inf\{\varepsilon\geq 0 \mid \B\models\sigma^\A_\varepsilon\}.
        \]
    \end{remark}

    \begin{definition}
        The $0^+$-Scott sentence $\sigma^\A_{0^+}$ of $\A$ is the sentence
        \begin{gather*}
            \bigwedge_{\varepsilon>0}\left( \gameformula{\varepsilon}{\innerscottrank(\A)}{\A} \land\vphantom{\bigwedge_{n<\omega}\bigwedge_{\bar{\varepsilon}\in\open{\varepsilon,\infty}^n}\bigwedge_{\bar{k}\in\omega^n}\bigwedge_{\varepsilon^+>\varepsilon}\bigwedge_{\bar{\varepsilon}^+\in\open{\varepsilon^+,\infty}^n}\bigwedge_{\bar{a}\in\dense{\A}^n}\forall v_0\dots\forall v_{n-1}\left( \gameformula{\varepsilon}{\innerscottrank(\A),\bar{\varepsilon},\bar{k}}{\A,\bar{a}} \to \gameformula{\varepsilon}{\innerscottrank(\A)+1,\bar{\varepsilon},\bar{k}}{\A,\bar{a}} \right)} \right. \\
            \left. \bigwedge_{n<\omega}\bigwedge_{\bar{\varepsilon}\in\open{\varepsilon,\infty}^n}\bigwedge_{\bar{k}\in\omega^n}\bigwedge_{\varepsilon^+>\varepsilon}\bigwedge_{\bar{\varepsilon}^+\in\open{\varepsilon^+,\infty}^n}\bigwedge_{\bar{a}\in\dense{\A}^n}\forall v_0\dots\forall v_{n-1}\left( \gameformula{\varepsilon}{\innerscottrank(\A),\bar{\varepsilon},\bar{k}}{\A,\bar{a}} \to \gameformula{\varepsilon^+}{\innerscottrank(\A)+1,\bar{\varepsilon}^+,\bar{k}}{\A,\bar{a}} \right) \right).
        \end{gather*}
    \end{definition}

    \begin{remark}
        As $\innerscottrank(\A)<\density(\A)^+$, for separable $\A$ we have $\sigma^\A_{0^+}\in\Logic{\omega_1}$.
    \end{remark}

    \begin{theorem}
        \label{Inner Scott sentence guarantees win for II in the long game}
        For $\A,\B\in\class$, the following are equivalent:
        \begin{enumerate}
            \item $\B\models\sigma^\A_{0^+}$,
            \item $\playertwo\wins\EF{\omega,\varepsilon}{\A,\B}{\A,\B}$ for all $\varepsilon>0$.
        \end{enumerate}
    \end{theorem}
    \begin{proof}
        Assume $\B\models\sigma^\A_{0^+}$. Denote $\alpha=\innerscottrank(\A)$ and let $\varepsilon>0$. We show that
        \[
            \playertwo\wins\EF{\omega,\varepsilon}{\A,\B}{\A,\B}.
        \]
        The strategy of $\playertwo$ is to play such moves $y_i$ that at all times the following holds:
        \begin{itemize}
            \item[$(*)$] If the position in $\EF{\omega,\varepsilon}{\A,\B}{\dense{\A},\B}$ is $((x_i,\varepsilon_i),y_i)_{i<n}$, then
            \[
                \B\models\gameformula{\delta^n}{\alpha,(\delta^n_0,\dots,\delta^n_{n-1}),(0,\dots,n-1)}{\A,(a_0,\dots,a_{n-1})}(b_0,\dots,b_{n-1}),
            \]
            where $\delta^n=\sum_{j<n+1}\varepsilon/2^{j+1}$ and $\delta^n_i = \sum_{j<n-i}\varepsilon_i/2^{j+1}$ for $i<n$.
        \end{itemize}
        Then the proof continues just like in Theorem~\ref{Theorem: Scott sentence guarantees win for II in the long game} but using the same idea as the proof of Theorem~\ref{Theorem: If II wins the Inner Scott rank + omega game she wins the long game}.

        Then assume that for all $\varepsilon>0$, $\playertwo\wins\EF{\omega,\varepsilon}{\A,\B}{\A,\B}$. Let $\varepsilon>0$. By Lemma~\ref{Theorem: If II wins the long game she wins all dynamic games} $\playertwo\wins\EFD{\innerscottrank(\A),\varepsilon}{\A,\B}{\A,\B}$. This means that $\B\models\gameformula{\varepsilon}{\innerscottrank(\A)}{\A}$. Then let $n$, $\bar{\varepsilon}$, $\bar{k}$, $\varepsilon^+$, $\bar{\varepsilon}^+$ and $\bar{a}\in\A^n$ and $\bar{b}\in\B^n$ be arbitrary, and suppose that $\B\models\gameformula{\varepsilon}{\scottrank_\varepsilon(\A),\bar{\varepsilon},\bar{k}}{\A,\bar{a}}(\bar{b})$. This means that
        \[
            \playertwo\wins\EFD{\innerscottrank(\A),\varepsilon,\bar{\varepsilon},\bar{k}}{\A,\B}{(\dense{\A},\bar{a}),(\B,\bar{b})}.
        \]
        By Theorem~\ref{Theorem: Inner Scott Rank works as intended}, we get that
        \[
            \playertwo\wins\EFD{\innerscottrank(\A)+1,\varepsilon^+,\bar{\varepsilon}^+,\bar{k}}{\A,\B}{(\dense{\A},\bar{a}),(\B,\bar{b})}.
        \]
        This means that $\B\models\gameformula{\varepsilon^+}{\innerscottrank(\A)+1,\bar{\varepsilon}^+,\bar{k}}{\A,\bar{a}}(\bar{b})$, showing that
        \[
            \B\models\gameformula{\varepsilon}{\scottrank_\varepsilon(\A),\bar{\varepsilon},\bar{k}}{\A,\bar{a}}(\bar{b})\to\gameformula{\varepsilon^+}{\innerscottrank(\A)+1,\bar{\varepsilon}^+,\bar{k}}{\A,\bar{a}}(\bar{b}).
        \]
        As $\bar{b}$ was arbitrary, we get that
        \[
            \B\models\forall v_0\dots\forall v_{n-1}\left( \gameformula{\varepsilon}{\scottrank_\varepsilon(\A),\bar{\varepsilon},\bar{k}}{\A,\bar{a}}\to\gameformula{\varepsilon^+}{\innerscottrank(\A)+1,\bar{\varepsilon}^+,\bar{k}}{\A,\bar{a}} \right),
        \]
        and since all the other parameters were also arbitrary, we can conclude that $\B$ satisfies the huge conjunction over all parameters. As $\varepsilon$ was also arbitrary, we finally get $\B\models\sigma^\A_{0^+}$.
    \end{proof}

    \begin{corollary}
        For separable $\A,\B\in\class$, $\B\models\sigma^\A_{0^+}$ if and only if there exists an $\varepsilon$-isomorphism $\A\to\B$ for all $\varepsilon>0$.
    \end{corollary}
    \begin{proof}
        Combine Theorems~\ref{Theorem: If II wins the long game then there exists an epsilon-isomorphism} and~\ref{Inner Scott sentence guarantees win for II in the long game}.
    \end{proof}

    \begin{remark}
        Let $\class$ be a class of metric structures and $\isom = (\isom_{\varepsilon})_{\varepsilon\geq 0}$ an isomorphism notion of $\class$. Then, if $L$ is a good vocabulary for $(\class,(\isom_{\varepsilon})_{\varepsilon\geq 0})$, we have
        \[
            \B\models\sigma^\A_{0^+} \iff d_{\isom}(\A,\B) = 0.
        \]
        Note that $d_{\isom}(\A,\B) = 0$ does not mean that $\A$ and $\B$ are isometric (but that they are \emph{almost isometric}).
    \end{remark}

    \subsection{Scott sentences for correspondence notions}

    \begin{definition}
        Let $\A\in\class$, $\varepsilon\geq 0$ and $\dense\A\subseteq\A$ a dense set of cardinality $<\kappa$. Given $n<\omega$ and $\bar{a}\in\A^n$, define the formulae $\gameformular{\varepsilon}{\alpha}{\A,\bar{a}}(v_0,\dots,v_{n-1})\in\Logic{\kappa}$, $\alpha\in\On$, by induction on $\alpha$ as follows:
        \begin{enumerate}
            \item for $\alpha = 0$,
            \begin{align*}
                \gameformular{\varepsilon}{\alpha}{\A,\bar{a}}(v_0,\dots,v_{n-1}) = \bigwedge\{& \appr(\varphi,\varepsilon)(v_{i_0}, \dots,v_{i_{k-1}}) \mid i_0,\dots,i_{k-1}<n, \\
                &\text{$\varphi$ atomic, $\A\models\varphi(a_{i_0},\dots,a_{i_{k-1}})$} \},
            \end{align*}
            \item for $\alpha=\beta+1$,
             \begin{align*}
                 \gameformular{\varepsilon}{\alpha}{\A,\bar{a}}(v_0,\dots,v_{n-1}) =& \left( \bigwedge_{a_n\in\dense\A}\exists v_n \gameformular{\varepsilon}{\beta}{\A,\concat{\bar{a}}{a_n}}(v_0,\dots,v_n) \right) \\
                 &{}\land\left( \forall v_n\bigvee_{a_n\in\dense\A}\gameformular{\varepsilon}{\beta}{\A,\concat{\bar{a}}{a_n}}(v_0,\dots,v_n) \right),
             \end{align*}
             and
            \item for $\alpha$ limit,
            \[
                \gameformular{\varepsilon}{\alpha}{\A,\bar{a}}(v_0,\dots,v_{n-1}) = \bigwedge_{\beta<\alpha}\gameformular{\varepsilon}{\beta}{\A,\bar{a}}(v_0,\dots,v_{n-1}).
            \]
        \end{enumerate}
    \end{definition}

    \begin{definition}
        The $\varepsilon$-Scott sentence $\varsigma_\varepsilon^\A$ of $\A\in\class$ is the sentence
        \[
            \gameformular{\varepsilon}{\scottrankr_\varepsilon(\A)}{\A} \land \bigwedge_{n<\omega}\bigwedge_{\bar{a}\in\dense{\A}^n}\forall v_0\dots\forall v_{n-1}\left( \gameformular{\varepsilon}{\scottrankr_\varepsilon(\A)}{\A,\bar{a}} \to \gameformular{\varepsilon}{\scottrankr_\varepsilon(\A)+1}{\A,\bar{a}} \right).
        \]
    \end{definition}

    \begin{definition}
        The $0^+$-Scott sentence $\varsigma_{0^+}^\A$ of $\A$ is the sentence
        \[
            \bigwedge_{\varepsilon>0}\left( \gameformular{\varepsilon}{\innerscottrankr(\A)}{\A} \land \bigwedge_{n<\omega}\bigwedge_{\varepsilon^+>\varepsilon}\bigwedge_{\bar{a}\in\dense\A^n}\forall v_0\dots\forall v_{n-1}\left( \gameformular{\varepsilon}{\innerscottrankr(\A)}{\A,\bar{a}} \to \gameformular{\varepsilon^+}{\innerscottrankr(\A)+1}{\A,\bar{a}} \right) \right).
        \]
    \end{definition}
    \begin{remark}
        As $\innerscottrankr(\A)<\density(\A)^+$, for separable $\A$ we have $\varsigma^\A_{0^+}\in\Logic{\omega_1}$.
    \end{remark}

    The following are proved in a fashion similar to their isomorphism notion counterparts.

    \begin{lemma}
        For $\A,\B\in\class$, and $\bar{a}\in\A^n$ and $\bar{b}\in\B^n$, the following are equivalent.
        \begin{enumerate}
            \item $\B\models\gameformular{\varepsilon}{\alpha}{\A,\bar{a}}(\bar{b})$.
            \item $\playertwo\wins\EFDR{\alpha,\varepsilon}{\A,\B}{(\dense\A,\bar{a}),(\B,\bar{b})}$.
        \end{enumerate}
    \end{lemma}

    \begin{theorem}
        For $\A,\B\in\class$, the following are equivalent.
        \begin{enumerate}
            \item $\B\models\varsigma_\varepsilon^\A$.
            \item $\playertwo\wins\EFR{\omega,\varepsilon}{\A,\B}{\dense\A,\B}$.
        \end{enumerate}
    \end{theorem}

    \begin{corollary}
        For separable $\A,\B\in\class$ and $\varepsilon>0$, the following are equivalent.
        \begin{enumerate}
            \item There is $\varepsilon'\in(0,\varepsilon)$ such that $\B\models\varsigma_{\varepsilon'}^\A$.
            \item There is $\varepsilon'\in(0,\varepsilon)$ and an $\varepsilon'$-correspondence between $\A$ and $\B$.
        \end{enumerate}
    \end{corollary}

    \begin{theorem}
        For $\A,\B\in\class$, the following are equivalent:
        \begin{enumerate}
            \item $\B\models\varsigma_{0^+}^\A$,
            \item $\playertwo\wins\EFR{\omega,\varepsilon}{\A,\B}{\dense\A,\B}$ for all $\varepsilon>0$.
        \end{enumerate}
    \end{theorem}

    \begin{corollary}
        For separable $\A,\B\in\class$, $\B\models\varsigma_{0^+}^\A$ if and only if there exists an $\varepsilon$-correspondence between $\A$ and $\B$ for all $\varepsilon>0$.
    \end{corollary}

    \begin{remark}
        Let $\class$ be a class of metric structures and $\corr = (\corr_{\varepsilon})_{\varepsilon\geq 0}$ a correspondence notion of $\class$. Then we have
        \[
            \B\models\varsigma^\A_{0^+} \iff d_{\corr}(\A,\B) = 0.
        \]
    \end{remark}

\end{document}